\documentclass[11pt,reqno]{amsart}
\usepackage{amsmath,amsthm,amssymb,mathrsfs,stmaryrd,color,iftex}
\usepackage[all]{xy}
\usepackage{url}
\usepackage{geometry}
\usepackage[utf8]{inputenc}
\usepackage[T1]{fontenc}

\usepackage{enumitem}
\usepackage[colorlinks=true,hyperindex, linkcolor=magenta, pagebackref=false, citecolor=cyan,pdfpagelabels]{hyperref}
\usepackage[capitalize]{cleveref}
\setlist[enumerate]{itemsep=2pt,parsep=2pt,before={\parskip=2pt}}

\newcommand{\cosimp}[3]{\xymatrix@1{#1 \ar@<.4ex>[r] \ar@<-.4ex>[r] & {\ }#2 \ar@<0.8ex>[r] \ar[r] \ar@<-.8ex>[r] & {\ } #3 \ar@<1.2ex>[r] \ar@<.4ex>[r] \ar@<-.4ex>[r] \ar@<-1.2ex>[r] & \cdots }}

\newcommand{\adjunction}[4]{\xymatrix@1{#1{\ } \ar@<0.3ex>[r]^{ {\scriptstyle #2}} & {\ } #3 \ar@<0.3ex>[l]^{ {\scriptstyle #4}}}}

\newtheorem{theorem}{Theorem}[section]
\newtheorem*{theorem*}{Theorem}
\newtheorem*{definition*}{Definition}
\newtheorem{proposition}[theorem]{Proposition}
\newtheorem{lemma}[theorem]{Lemma}
\newtheorem{corollary}[theorem]{Corollary}

\theoremstyle{definition}
\newtheorem{definition}[theorem]{Definition}

\newtheorem{remark}[theorem]{Remark}

\newtheorem{example}[theorem]{Example}

\crefname{assumption}{assumption}{assumptions}
\crefname{construction}{construction}{constructions}

\newcommand{\et}{\mathrm{\acute{e}t}}

\newcommand{\Hom}{\mathrm{Hom}}
\newcommand{\sHom}{\mathcal{H}\mathrm{om}}

\newcommand{\Spec}{\mathrm{Spec}}

\newcommand{\proet}{\mathrm{pro\acute{e}t}}
\newcommand{\cons}{\mathrm{cons}}
\newcommand{\tor}{\mathrm{tor}}
\newcommand{\dotimes}{\otimes^{\mathbb L}}

\begin{document}

\title{Relative perversity}
\author{David Hansen and Peter Scholze}

\begin{abstract} We define and study a relative perverse $t$-structure associated with any finitely presented morphism of schemes $f: X\to S$, with relative perversity equivalent to perversity of the restrictions to all geometric fibres of $f$. The existence of this $t$-structure is closely related to perverse $t$-exactness properties of nearby cycles. This $t$-structure preserves universally locally acyclic sheaves, and one gets a resulting abelian category $\mathrm{Perv}^{\mathrm{ULA}}(X/S)$ with many of the same properties familiar in the absolute setting (e.g., noetherian, artinian, compatible with Verdier duality). For $S$ connected and geometrically unibranch with generic point $\eta$, the functor $\mathrm{Perv}^{\mathrm{ULA}}(X/S)\to \mathrm{Perv}(X_\eta)$ is exact and fully faithful, and its essential image is stable under passage to subquotients. This yields a notion of ``good reduction'' for perverse sheaves.
\end{abstract}

\maketitle

\tableofcontents

\section{Introduction}

Since their discovery over forty years ago \cite{BBDG}, perverse sheaves quickly emerged as objects of fundamental importance in topology, algebraic geometry, and representation theory. On one hand, perverse sheaves are the correct generalization of local systems from smooth spaces to singular spaces. This is well illustrated by the Poincar\'e duality results for intersection cohomology, and the celebrated decomposition theorem of Beilinson-Bernstein-Deligne-Gabber. On the other hand, many naturally occurring spaces in geometric representation theory are singular, and perverse sheaves on them give rise to a rich source of constructions in representation theory. We refer the reader to the article \cite{dCMBulletin} for a beautiful overview of perverse sheaves and their many applications.

Notably, perverse sheaves are an ``absolute'' theory. Nevertheless, there are hints in the literature that some kind of relative version of this theory should exist. For example, in the geometric Satake equivalence, one studies perverse sheaves on the affine Grassmannian, and at various points of the argument it is important to deform the affine Grassmannian into a family of affine Grassmannians known as the Beilinson--Drinfeld Grassmannian, and correspondingly deform the perverse sheaves into families of perverse sheaves. In particular, in \cite{FarguesScholze}, a ``relative perverse $t$-structure'' is introduced on the Beilinson--Drinfeld Grassmannian, where objects of the heart are those complexes that are perverse after restriction to each fibre of the family.

More generally, in any application of perverse sheaves, one may wonder how these perverse sheaves vary when some of the geometric objects under consideration vary -- be it a smooth projective curve over which objects are defined, or a vector bundle over it, etc. In other words, roughly speaking, one has some base space $S$ and for each $s\in S$ a space $X_s$ on which one wants to consider perverse sheaves $A_s$. As usual the collection $X_s$, $s\in S$, is encoded in a morphism $X\to S$ with fibres $X_s$, and the family of perverse sheaves $A_s$ should assemble into an object $A\in D(X)$ in the derived category of sheaves on $X$. However, there was no adequate language to talk about such ``families of perverse sheaves''. Namely, $A$ will not be a perverse sheaf on the total space $X$. What we will show, however, is that there is a $t$-structure on the derived category of sheaves on $X$ whose heart consists of those $A$ for which all the fibres $A_s$ on $X_s$ are perverse. In other words, the main theorem of this paper is that for any family $f: X\to S$, there is a $t$-structure on the derived category of $\ell$-adic sheaves on $X$ such that $A\in D(X)$ lies in the heart if and only if $A|_{X_{\overline{s}}}\in D(X_{\overline{s}})$ is perverse for all geometric points $\overline{s}$ of $S$. There are two extreme situations: If $S$ is a point, this recovers the usual ``absolute'' perverse $t$-structure; and if $X=S$ then this recovers the standard $t$-structure (whose heart are usual sheaves).

What is surprising is that the existence of such a $t$-structure is both true and non-tautological; in fact, its existence is closely related to perverse $t$-exactness of the nearby cycles functor. While we do not give any real applications in this paper, we do expect that this result will be useful whenever one is studying certain phenomena in families, when those phenomena involve perverse sheaves.

To state our results more precisely, fix a prime $\ell$. We assume that all schemes are qcqs, and live over $\mathbb Z[\tfrac 1\ell]$. Let $f: X\to S$ be a morphism of finite presentation between such schemes. The goal of this paper is to introduce a ``relatively (over $S$) perverse $t$-structure'' on the derived category of \'etale sheaves on $X$, and show that it interacts well with the notion of universally locally acyclic sheaves.

Although probably the case of $\overline{\mathbb Q}_\ell$-coefficients is the most interesting, we also allow some other coefficients; in particular, the case of torsion coefficients is required as an intermediate step.\footnote{As we will apply a lot of descent techniques, we prefer to work with $\infty$-categories. However, as $t$-structures only depend on the underlying triangulated category, the statements of our main results are really about the underlying triangulated categories.}

\begin{enumerate}
\item[{\rm (A)}] Let $\Lambda$ be a ring killed by some power of $\ell$, and denote by $\mathcal D_\et(X,\Lambda)$ the left-completion of the derived $\infty$-category $\mathcal D(X_\et,\Lambda)$ of $\Lambda$-modules on the \'etale site $X_\et$. (If $X_\et$ has locally finite $\ell$-cohomological dimension, then the left-completion is not necessary.)
\item[{\rm (B)}] In the setting of (A), let $\mathcal D_\cons(X,\Lambda)\subset \mathcal D_\et(X,\Lambda)$ be the full $\infty$-subcategory of perfect-constructible complexes. (If $X_\et$ has locally finite $\ell$-cohomological dimension, then $\mathcal D_\et(X,\Lambda)$ is compactly generated with compact objects $\mathcal D_\cons(X,\Lambda)$.)
\item[{\rm (C)}] Let $\Lambda$ be an algebraic extension $L$ of $\mathbb Q_\ell$ or its ring of integers $\mathcal O_L$, and let $\mathcal D_\cons(X,\Lambda)$ be defined as in \cite{HemoRicharzScholbach}. In other words, it is the full $\infty$-subcategory of $\mathcal D(X_\proet,\Lambda)$ consisting of those objects that on a constructible stratification of $X$ become dualizable; by \cite{HemoRicharzScholbach} this agrees with more classical definitions.
\end{enumerate}

In the respective cases, we let $\mathcal D(X)=\mathcal D_\et(X,\Lambda)$, resp.~$\mathcal D(X)=\mathcal D_\cons(X,\Lambda)$, resp.~$\mathcal D(X)=\mathcal D_\cons(X,\Lambda)$. In all cases, $\mathcal D(X)$ is a $\Lambda$-linear $\infty$-category, and pullback along any map $f: Y\to X$ defines functors $\mathcal D(X)\to \mathcal D(Y)$. In fact, $\mathcal D(X)$ is naturally a full $\infty$-subcategory of $\mathcal D(X_\proet,\Lambda)$ stable under pullbacks. In all cases, there is also a symmetric monoidal tensor product, and pullback commutes with tensor products.

In setting (B), it is sometimes important to assume that $\Lambda$ is regular as otherwise this category is not stable under naive truncations. The precise condition on $\Lambda$ we require is that any truncation of a perfect complex of $\Lambda$-modules is still perfect, so whenever we ask that $\Lambda$ is regular, we mean this condition.

If $f: Y\to X$ is separated and of finite type (resp.~of finite presentation in cases (B) and (C)), there is a natural functor $Rf_!: \mathcal D(Y)\to \mathcal D(X)$ compatible with base change and satisfying a projection formula. In case (A), the adjoint functor theorem also gives us right adjoints, and thus internal Hom's, direct images, and exceptional inverse images, and these may or may not preserve subcategories of constructible complexes in general.

The main theorem of the paper is the following.\footnote{As far as we are aware, this notion of relative perversity is new, but in some restricted variant the notion has been considered before by Katz--Laumon \cite{KatzLaumon}.}

\begin{theorem}\label{thm:main} Let $D(X)$ denote the derived category of $\Lambda$-modules in any of the settings (A), (B), and (C). In case (B), assume that $\Lambda$ is regular. In case (C), assume that any constructible subset of $S$ has finitely many irreducible components.

There is a (necessarily unique) $t$-structure $({}^{p/S} D^{\leq 0},{}^{p/S}  D^{\geq 0})$ on $ D(X)$, called the relative perverse $t$-structure, with the following property:

An object $A\in D(X)$ lies in ${}^{p/S} D^{\leq 0}$ (resp.~${}^{p/S} D^{\geq 0}$) if and only if for all geometric points $\overline{s}\to S$ with fibre $X_{\overline{s}} = X\times_S \overline{s}$, the restriction $A|_{X_{\overline{s}}}\in  D(X_{\overline{s}})$ lies in ${}^p  D^{\leq 0}$ (resp.~${}^p  D^{\geq 0}$), for the usual (absolute) perverse $t$-structure.
\end{theorem}

\begin{remark} The hypotheses (in cases (B), (C)) are essentially optimal, see Remark~\ref{rem:hypotheses}.
\end{remark}

\begin{remark} Another setting of interest is when $X$ and $S$ are of finite type over $\mathbb C$, in which case one can also use constructible sheaves with $\mathbb Z$-, $\mathbb Q$-, or $\mathbb C$-coefficients. The theorem and its variants discussed below also holds true in that setting, and can be deduced from their $\ell$-adic versions. 
\end{remark}

Again, the existence of this t-structure is somewhat unexpected. More precisely, it is totally formal (at least in setting (A)) that there is a t-structure on $D(X)$ whose connective part is given by complexes which are perverse connective on every geometric fiber. However, the coconnective part of this t-structure is completely inexplicit, and it is very surprising that it turns out to admit such a clean fiberwise description.

The proof of this theorem rests on two ingredients: v-descent, and the theory of nearby cycles. Roughly speaking, v-descent allows us to reduce to the case that $S$ is the spectrum of a valuation ring $V$ with algebraically closed fraction field. In that case the theorem is closely related to the perverse $t$-exactness properties of nearby cycles.

Let us first state the results regarding v-descent; the results here are mostly due to Bhatt--Mathew \cite{BhattMathew} who even prove arc-descent, and their results have been further refined by Gabber \cite{GabberLetterMathew}. In particular, there is no claim of originality in this part. Recall that a map of qcqs schemes $f: Y\to X$ is a v-cover if for any map $\Spec V\to X$ from a valuation ring $V$, there is a faithfully flat extension $V\subset W$ of valuation rings and a lift $\Spec W\to Y$. This is an extremely general class of covers. Even more general is the class of arc-covers, where this lifting condition is restricted to valuation rings of rank $\leq 1$. Intermediate between v-covers and arc-covers is the notion of universal submersions; these are the maps $f: Y\to X$ such that any base change of $f$ induces a quotient map on topological spaces. It is equivalent to the condition that for any map $\Spec V\to X$ as above, with fraction field $K$ of $V$, the inclusion $Y_K\subset Y_V$ is not closed.

\begin{theorem}[Bhatt--Mathew \cite{BhattMathew}, Gabber \cite{GabberLetterMathew}] In any of the settings (A), (B), and (C), the association $X\mapsto \mathcal D(X)$ defines a v-sheaf of $\infty$-categories. In fact, in settings (B) and (C) it is even an arc-sheaf of $\infty$-categories, and in setting (A) a sheaf for universal submersions.
\end{theorem}

We warn the reader that in setting (A) it is not an arc-sheaf, by an example of Gabber, see Example~\ref{ex:arcnotsubmersion}. In fact, universal submersions are the most general class of maps that one can allow. The key step in Gabber's proof is worth stating separately, as it is about general \'etale sheaves (without abelian group structure).

\begin{theorem}[\cite{GabberLetterMathew}] Sending any scheme $X$ to the category of \'etale sheaves on $X$ defines a stack with respect to universal submersions.\footnote{In \cite{GabberLetterMathew}, Gabber also sketches an extension of this result to the case where one sends $X$ to the $(2,1)$-category of ind-finite \'etale stacks.}
 In particular, sending any scheme $X$ to the category of separated \'etale maps of schemes $Y\to X$ defines a stack with respect to universal submersions, and in particular a v-stack.
\end{theorem}

This strengthens some previous descent results, notably by Rydh \cite{Rydhvdescent}, \cite[Theorem 5.6]{BhattMathew}.

Using these descent results and some approximation arguments, we can reduce Theorem~\ref{thm:main} to the case that $S=\Spec V$ where $V$ is a valuation ring with algebraically closed fraction field $K$; one can even assume that $V$ is of rank $1$.

In that case, we rely on the theory of nearby cycles. The foundational results here are due to Deligne \cite{SGA412}, Illusie and Gabber \cite[Appendix]{IllusieAutour}, Huber \cite[Section 4.2]{HuberBook}, Zheng \cite[Appendix]{ZhengDuality}, and recently Lu--Zheng \cite{LuZhengULA}. We take the opportunity to rederive all the basic results about nearby cycles from the perspective of the notion of universal local acyclicity, using critically the recent characterization of universal local acyclicity in terms of dualizability in a symmetric monoidal $2$-category of cohomological correspondences, due to Lu--Zheng \cite{LuZhengULA}. Again, there is no claim of originality.

This symmetric monoidal $2$-category can be defined in any of the settings (A), (B), and (C), but it turns out that universal local acyclicity (i.e., dualizability in this category) implies constructibility, so settings (A) and (B) yield the same universally locally acyclic objects. For this reason, we restrict to settings (B) and (C) for the moment.

\begin{theorem} Let $f: X\to S$ be a separated map of finite presentation between qcqs schemes and let $A\in D(X)$ in one of the settings (B) and (C). The following conditions are equivalent.
\begin{enumerate}
\item[{\rm (i)}] The pair $(X,A)$ defines a dualizable object in the symmetric monoidal $2$-category of cohomological correspondences over $S$.
\item[{\rm (ii)}] The following condition holds after any base change in $S$. For any geometric point $\overline{x}\to X$ mapping to a geometric point $\overline{s}\to S$, and a generization $\overline{t}\to S$ of $\overline{s}$, the map
\[
A|_{\overline{x}} = R\Gamma(X_{\overline{x}},A)\to R\Gamma(X_{\overline{x}}\times_{S_{\overline{s}}} S_{\overline{t}},A)
\]
is an isomorphism.
\item[{\rm (iii)}] The following condition holds after any base change in $S$. For any geometric point $\overline{x}\to X$ mapping to a geometric point $\overline{s}\to S$, and a generization $\overline{t}\to S$ of $\overline{s}$, the map
\[
A|_{\overline{x}} = R\Gamma(X_{\overline{x}},A)\to R\Gamma(X_{\overline{x}}\times_{S_{\overline{s}}} \overline{t},A)
\]
is an isomorphism.
\item[{\rm (iv)}] After base change along $\Spec V\to S$ for any rank $1$ valuation ring $V$ with algebraically closed fraction field $K$ and any geometric point $\overline{x}\to X$ mapping to the special point of $\Spec V$, the map
\[
A|_{\overline{x}} = R\Gamma(X_{\overline{x}},A)\to R\Gamma(X_{\overline{x}}\times_{\Spec V}\Spec K,A)
\]
is an isomorphism.
\end{enumerate}
Moreover, these conditions are stable under any base change, and can be checked arc-locally on $S$.
\end{theorem}

In particular, this shows that the key to understanding universal local acyclicity is the case where the base is the spectrum of a (rank $1$) valuation ring with algebraically closed fraction field. The key result is the following, which rederives all the basic properties of the nearby cycles functor.

\begin{theorem} Let $X$ be a separated scheme of finite presentation over $S=\Spec V$, where $V$ is a valuation ring with algebraically closed fraction field $K$. Let $j: X_K\subset X$ be the inclusion of the generic fibre. Then, in the settings (B) and (C), the restriction functor
\[
j^\ast: D^{\mathrm{ULA}}(X/S)\to D(X_K)
\]
is an equivalence, whose inverse is given by $Rj_\ast: D(X_K)\subset D(X_{K,\proet},\Lambda)\to D(X_\proet,\Lambda)$.

In particular, the formation of $Rj_\ast$ preserves constructibility, and commutes with any flat base change $V\to V'$ of valuation rings with algebraically closed fraction fields, with relative Verdier duality, and satisfies a K\"unneth formula.
\end{theorem}

Given a separated map $f: X\to S$ of finite presentation, the functor taking $S'/S$ to $\mathcal D^{\mathrm{ULA}}(X_{S'}/S')$ has good properties.

\begin{proposition} In any setting, the functor $S'\mapsto \mathcal D^{\mathrm{ULA}}(X_{S'}/S')$ is an arc-sheaf of $\infty$-categories. Moreover, it satisfies the valuative criterion of properness in the sense that if $S'=\Spec V$ is the spectrum of a valuation ring $V$ with algebraically closed fraction field $K$, then $\mathcal D^{\mathrm{ULA}}(X_V/V)\to \mathcal D^{\mathrm{ULA}}(X_K/K)$ is an equivalence. In setting (B), it is a finitary arc-sheaf.

In case (C), let $L$ be the algebraic extension of $\mathbb Q_\ell$ and fix some $A\in \mathcal D^{\mathrm{ULA}}(X/S,L)$. Consider the functor taking $S'/S$ to the $\infty$-category of all $A_0\in \mathcal D^{\mathrm{ULA}}(X_{S'}/S',\mathcal O_L)$ with an identification $A_0[\tfrac 1{\ell}]\cong A|_{X_{S'}}$. This is a finitary arc-sheaf satisfying the valuative criterion of properness, and is v-locally nonempty.
\end{proposition}

The second part of this proposition shows that at least h-locally on the base (where h-covers are by definition finitely presented v-covers), any universally locally acyclic sheaf with rational coefficients admits an integral structure that is also universally locally acyclic; through this result one can get a handle on the case of rational coefficients. We note that Proposition~\ref{prop:ULAintegralstructureunibranch} shows that if $S$ is geometrically unibranch, such an integral structure exists already over $S$.

Moreover, relative perversity interacts well with universal local acyclicity. More precisely:

\begin{theorem}\label{thm:ULAmain} Assume that $X$ is a separated scheme of finite presentation over $S$, and consider one of the settings (B) and (C). In case (B), assume that $\Lambda$ is regular. In case (C), assume that $S$ has only finitely many irreducible components. Then there is a relative perverse $t$-structure
\[
{}^{p/S} D^{\mathrm{ULA},\leq 0}(X/S),{}^{p/S} D^{\mathrm{ULA},\geq 0}(X/S)\subset D^{\mathrm{ULA}}(X/S)
\]
such that $A\in {}^{p/S} D^{\mathrm{ULA},\leq 0}(X/S)$ (resp.~$A\in {}^{p/S} D^{\mathrm{ULA},\geq 0}(X/S)$) if and only if for all geometric points $\overline{s}\to S$, the fibre $A|_{X_{\overline{s}}}$ lies in ${}^p D^{\leq 0}(X_{\overline{s}})$ (resp.~${}^p D^{\geq 0}(X_{\overline{s}})$).
\end{theorem}

In particular, the inclusion $D^{\mathrm{ULA}}(X/S)\subset D(X)$ is $t$-exact for the relative perverse $t$-structure, and thus for any $A\in D^{\mathrm{ULA}}(X/S)$ its relative perverse cohomologies ${}^{p/S} \mathcal H^i(A)$ are again universally locally acyclic over $S$.

If $S$ is regular of equidimension $d$, then we can also equip $D(X)$ with an absolute perverse $t$-structure. In that case, the shifted inclusion
\[
D^{\mathrm{ULA}}(X/S)\subset D(X): A\mapsto A[d]
\]
is $t$-exact. Thus, in this case the absolute perverse cohomologies ${}^p \mathcal H^i(A)$ are again universally locally acyclic over $S$. This generalizes a result of Gaitsgory \cite{GaitsgoryULA} who proved this result when $S$ is assumed to be smooth over a field.

By Theorem~\ref{thm:ULAmain}, we get in particular a well-behaved ($\Lambda$-linear) abelian category $\mathrm{Perv}^{\mathrm{ULA}}(X/S)$ of relatively perverse universally locally acyclic sheaves over $S$. Our final result concerns properties of this abelian category.

\begin{theorem}\label{thm:perverseULAmain} Consider one of the settings (B) and (C). In case (B), assume that $\Lambda$ is regular. Moreover, in all settings, assume that $S$ is irreducible, and let $\eta\in S$ be the generic point, with $j: X_\eta\subset X$ the inclusion.
\begin{enumerate}
\item[{\rm (i)}] The restriction functor
\[
j^\ast: \mathrm{Perv}^{\mathrm{ULA}}(X/S)\to \mathrm{Perv}(X_\eta)
\]
is an exact and faithful functor of abelian categories. If $\Lambda$ is noetherian, the category $\mathrm{Perv}^{\mathrm{ULA}}(X/S)$ is noetherian. If $\Lambda$ is artinian, it is also artinian.
\item[{\rm (ii)}] Assume that $S$ is geometrically unibranch. The restriction functor
\[
j^\ast: \mathrm{Perv}^{\mathrm{ULA}}(X/S)\to \mathrm{Perv}(X_\eta)
\]
is exact and fully faithful, and its image is stable under subquotients.
\end{enumerate}
\end{theorem}

We note that the fully faithfulness in part (ii) is a strengthening of a theorem of Reich \cite[Proposition IV.2.8]{ReichExtendULA} who essentially proved the case that $S$ is smooth over a field. Results of this type are used in the proof of the geometric Satake equivalence, which involves an analysis of perverse universally locally acyclic sheaves on Beilinson--Drinfeld Grassmannians. In particular, one needs to know that these are determined by their restriction to a dense open subset of the base, in order to construct the fusion product. Part (ii) gives a very general result of this form. We note that the hypothesis in part (ii) is necessary already when $X=S$, in which case one is looking at local systems on $S$.

\begin{remark} Part (ii) can be seen as giving a notion of ``good reduction'' for a perverse sheaf: If say $S=\Spec\, \mathbb Z_p$ and $X/S$ is a scheme of finite type and $A_0\in \mathrm{Perv}(X_{\mathbb Q_p})$ is a perverse sheaf on the generic fibre, we can ask whether $A_0$ ``has good reduction'' in the sense of extending to a (necessarily relatively perverse) universally locally acyclic sheaf on $X/S$. In that case, its special fibre agrees with the nearby cycle sheaf, so the action of the absolute Galois group of $\mathbb Q_p$ on the nearby cycles is unramified. In fact, the converse is also true. However, over higher-dimensional bases, the condition is more subtle. Let us remark that we have not investigated the relation to the theory of nearby cycles over higher-dimensional base schemes.
\end{remark}

We have the following corollary, which again recovers and extends a result of Gaitsgory \cite{GaitsgoryULA} (who treated the case where $S$ is a smooth variety over a field).

\begin{corollary} Assume that $S$ is regular and connected, of dimension $d$, and that $\Lambda$ is a field. Assume that $A\in D(X)$ is absolutely perverse, and universally locally acyclic over $S$. Then any absolutely perverse subquotient of $A$ is universally locally acyclic over $S$.
\end{corollary}

\begin{proof} The generic fibre $A_\eta$ admits a finite Jordan--H\"older filtration, which by Theorem~\ref{thm:perverseULAmain}~(ii) extends to a filtration of $A$ by universally locally acyclic sheaves that are absolutely perverse (as over a regular base, absolute and relative perversity agree up to shift for universally locally acyclic sheaves). We can thus assume that $A_\eta$ is simple. In that case one sees that $A$ is also necessarily simple: Indeed, its restriction to a smooth locally closed subscheme of $S$ is still relatively perverse up to shift by dimension $d$, and thus with respect to absolute perversity it lies in ${}^p \mathcal D^{\leq -1}$; and the same argument applies to its Verdier dual.
\end{proof}

{\bf Acknowledgments.} Several years ago, DH conjectured the existence of a well-behaved theory of ``families of perverse sheaves'', motivated by the geometric Langlands literature and in particular Gaitsgory's unpublished notes \cite{GaitsgoryULA}, and gave a talk about these ideas in Bonn in June 2019. During the writing of the proof of the geometric Satake equivalence in \cite{FarguesScholze}, PS realized that the desired theory would follow from the existence of a general relative perverse $t$-structure. The authors were then quickly able to cook up a proof by combining v-descent with the theory of nearby cycles. We thank Akhil Mathew for sharing with us Gabber's letter \cite{GabberLetterMathew}. Moreover, we thank Dennis Gaitsgory, Haoyu Hu, Gerard Laumon, Akhil Mathew, Timo Richarz, Longke Tang, Enlin Yang and Weizhe Zheng for discussions and feedback. Finally, we thank Owen Barrett, Patrick Bieker and the referee for helpful comments and corrections. During the writing of this manuscript, Scholze was supported by a DFG Leibniz Prize, and by the DFG under the Excellence Strategy – EXC-2047/1 – 390685813. 

\section{Derived categories of \'etale sheaves}

In this section, we recall some basics on derived categories of \'etale sheaves. As in the introduction, we consider one of three settings.

\begin{enumerate}
\item[{\rm (A)}] Let $\Lambda$ be a ring killed by some power of $\ell$, and denote by $\mathcal D_\et(X,\Lambda)$ the left-completion of the derived $\infty$-category $\mathcal D(X_\et,\Lambda)$ of $\Lambda$-modules on the \'etale site $X_\et$. (If $X_\et$ has locally finite $\ell$-cohomological dimension, then the left-completion is not necessary.)
\item[{\rm (B)}] In the setting of (A), let $\mathcal D_\cons(X,\Lambda)\subset \mathcal D_\et(X,\Lambda)$ be the full $\infty$-subcategory of perfect-constructible complexes. (If $X_\et$ has locally finite $\ell$-cohomological dimension, then $\mathcal D_\et(X,\Lambda)$ is compactly generated with compact objects $\mathcal D_\cons(X,\Lambda)$.)
\item[{\rm (C)}] Let $\Lambda$ be an algebraic extension $L$ of $\mathbb Q_\ell$ or its ring of integers $\mathcal O_L$, and let $\mathcal D_\cons(X,\Lambda)$ be the full $\infty$-subcategory of $\mathcal D(X_\proet,\Lambda)$ consisting of those objects that on a constructible stratification of $X$ become dualizable.
\end{enumerate}

We will discuss each setting in turn, and discuss the definition of the pullback, tensor, and proper pushforward functors. We start with settings (A) and (B). The starting point is the following proposition.

\begin{proposition}[{\cite[Proposition 5.3.2]{BhattScholzeProetale}}]\label{prop:pullbackproetale} Let $X$ be a qcqs scheme and let $\Lambda$ be any ring. Let $\nu_X: X_\proet\to X_\et$ be the projection from the pro-\'etale site of $X$ to the \'etale site of $X$. Then
\[
\nu_X^\ast: \mathcal D^+(X_\et,\Lambda)\to \mathcal D^+(X_\proet,\Lambda)
\]
is fully faithful, and it extends to a fully faithful functor
\[
\nu_X^\ast: \mathcal D_\et(X,\Lambda)\to \mathcal D(X_\proet,\Lambda)
\]
from the left-completion $\mathcal D_\et(X,\Lambda)$ of $\mathcal D(X_\et,\Lambda)$. The essential image of $\nu_X^\ast$ is the full $\infty$-subcategory of all $A\in \mathcal D(X_\proet,\Lambda)$ such that for all $i\in \mathbb Z$, the pro-\'etale sheaf $\mathcal H^i(A)$ comes via pullback from the \'etale site.
\end{proposition}

Moreover, if $f: X\to S$ is a separated map of finite type, then choosing a compactification $\overline{f}: \overline{X}\to S$, $j: X\hookrightarrow \overline{X}$, we can define
\[
Rf_! = R\overline{f}_\ast j_!: \mathcal D_\et(X,\Lambda)\to \mathcal D_\et(S,\Lambda).
\]
It follows from the usual formalism that this functor is independent of the choice of compactification, preserves all colimits, commutes with pullbacks, and satisfies a projection formula. As $\mathcal D_\et(X,\Lambda)$ is a presentable $\infty$-category, one can also use the adjoint functor theorem to see that there are functors $R\sHom_\Lambda$, $Rf_\ast$ and $Rf^!$ right adjoint to $\dotimes_\Lambda$, $f^\ast$ and $Rf_!$, satisfying all the usual formalism. (We do not try to make the $6$-functor formalism into a coherent $\infty$-categorical structure here; all coherences between these operations are only claimed as data on the level of homotopy categories.)

For setting (B), we restrict to the full $\infty$-subcategory
\[
\mathcal D_\cons(X,\Lambda)\subset \mathcal D_\et(X,\Lambda)
\]
of constructible objects, i.e.~objects that become locally constant with perfect fibres over a constructible stratification. Again, this is stable under tensor products and pullbacks, and if $f: X\to S$ is separated and of finite presentation, then $Rf_!$ preserves $\mathcal D_\cons(X,\Lambda)$ by the usual finiteness results.

In setting (C), we first define, following \cite{HemoRicharzScholbach}
\[
\mathcal D_\cons(X,\Lambda)\subset \mathcal D(X_\proet,\Lambda)
\]
as the full $\infty$-subcategory of all objects $A$ such that $A|X_i \in \mathcal D(X_{i,\proet},\Lambda)$ is dualizable for some constructible stratification $X = \cup_i X_i$. This definition agrees with the more classical definition. Namely, \cite{HemoRicharzScholbach} show that
\[
\mathcal D_\cons(X,L) = \varinjlim_{L'\subset L} \mathcal D_\cons(X,L'),\ \mathcal D_\cons(X,\mathcal O_L) = \varinjlim_{L'\subset L} \mathcal D_\cons(X,\mathcal O_{L'})
\]
as $L'\subset L$ ranges over finite extensions of $\mathbb Q_\ell$, reducing the study of these $\infty$-categories to the case of finite extensions $L$ of $\mathbb Q_\ell$. (In fact, this follows quickly from the definitions.) In that case the functor
\[
\mathcal D_\cons(X,\mathcal O_L)\to \varprojlim_n \mathcal D_\cons(X,\mathcal O_L/\ell^n)
\]
is an equivalence (again, this is not hard to prove), and we will show below that
\[
\mathcal D_\cons(X,L) = \mathcal D_\cons(X,\mathcal O_L)\otimes_{\mathcal O_L} L.
\]
Here, it is easy to see that the natural functor
\[
\mathcal D_\cons(X,\mathcal O_L)\otimes_{\mathcal O_L} L\to \mathcal D_\cons(X,L)
\]
is fully faithful.

It follows from the definition that $\mathcal D_\cons(X,\Lambda)\subset \mathcal D(X_\proet,\Lambda)$ is a symmetric monoidal $\infty$-subcategory, compatible with pullback along $f: Y\to X$. The other description of $\mathcal D_\cons(X,\Lambda)$ also shows that one can define a functor $Rf_!$ for separated maps of finite presentation $f: X\to S$, via reduction to case (B); they continue to satisfy all the usual properties.

These $\infty$-categories of constructible objects satisfy arc-descent.

\begin{theorem}\label{thm:arcdescentD} In settings (B) and (C), the functor $X\mapsto \mathcal D_\cons(X,\Lambda)$ defines an arc-sheaf of $\infty$-categories. It is a finitary arc-sheaf in setting (B).
\end{theorem}

\begin{proof} In setting (B), this result is due to Bhatt--Mathew \cite[Theorem 5.4, Theorem 5.13]{BhattMathew}, at least when $\Lambda$ is finite. Their \cite[Theorem 5.4]{BhattMathew} applies in general, as does the argument that it is a finitary presheaf, so it remains to establish effectivity of descent. This will be done later in the full setting (A) for universal submersions, to which we can reduce by noetherian approximation.

In setting (C), one can formally reduce to the case of finite extensions $L/\mathbb Q_\ell$. In that case, the case of $\mathcal O_L$-coefficients follows via passage to limits from setting (B). This also formally implies that with $L$-coefficients, for any arc-cover $Y\to X$ with Cech nerve $Y^{\bullet/X}$, the map
\[
\mathcal D_\cons(X,L)\to \lim_\Delta \mathcal D_\cons(Y^{\bullet/X},L)
\]
is fully faithful. It remains to show effectivity of descent. For this, we first prove the following result regarding the existence of $\mathcal O_L$-lattices.

\begin{proposition}\label{prop:constructiblelattices} Let $L$ be an algebraic extension of $\mathbb Q_\ell$ and fix $A\in \mathcal D_\cons(X,L)$. Consider the functor taking an $X$-scheme $X'$ to the $\infty$-category of $A_0\in \mathcal D_\cons(X',\mathcal O_L)$ together with an identification $A_0[\tfrac 1\ell]\cong A|_{X'}$. This functor is a finitary arc-sheaf, and admits a section over an \'etale cover of $X$.
\end{proposition}

\begin{proof} It is clear that it is an arc-sheaf. To construct a section over an \'etale cover, one reduces to the case that $A$ is dualizable. In that case we can arrange that $A_0$ is also dualizable. Over a w-contractible pro-\'etale cover $X'\to X$, the complex $A_0$ is then equivalent to a perfect complex of $C(\pi_0 X',L)$-modules, cf.~\cite{HemoRicharzScholbach}. But as $\pi_0 X'$ is extremally disconnected, any finitely generated ideal of $C(\pi_0 X',L)$ is principal and isomorphic as $C(\pi_0 X',L)$-module to a direct summand of $C(\pi_0 X',L)$ generated by an idempotent. Any such idempotent is necessarily integral, from which one can deduce that any perfect complex of $C(\pi_0 X',L)$-modules can be extended to a perfect complex of $C(\pi_0 X',\mathcal O_L)$-modules (see \cite[Corollary 3.39]{HemoRicharzScholbach} for a closely related result). This then gives an integral structure over $X'$, and by finitaryness, this section over $X'$ is already defined over an \'etale cover of $X$.

It remains to prove that it is finitary. It is enough to do this arc-locally. By the first paragraph, we can always find a section over a pro-\'etale cover, so we can assume that $A=A_1[\tfrac 1\ell]$ for some $A_1\in D_\cons(X,\mathcal O_L)$ that we fix. Now let $X'=\varprojlim_i X_i'$ be an inverse limit of affine $X$-schemes $X_i'=\Spec R_i$. We want to show that the functor
\[
\varinjlim_i \{A_{0,i}\in \mathcal D_\cons(X_i',\mathcal O_L), A_{0,i}[\tfrac 1\ell]\cong A|_{X_i'}\}\to \{A_0\in \mathcal D_\cons(X',\mathcal O_L), A_0[\tfrac 1\ell]\cong A|_{X'}\}
\]
is an equivalence. First, it is fully faithful: For this, we can fix some $i$ and $A_{0,i}, A_{0,i}'\in \mathcal D_\cons(X_i',\mathcal O_L)$ with
\[
A_{0,i}[\tfrac 1\ell]\cong A|_{X_i'}\cong A_{0,i}'[\tfrac 1\ell].
\]
There is a map $A_{0,i}\to A_{0,i}'$ commuting with the isomorphism after inverting $\ell$ if and only if the map from $A_{0,i}$ to the cone of $A_{0,i}'\to A|_{X_i'}$ vanishes; and if such a map exists, then the set of maps forms a torsor under the maps from $A_{0,i}$ to the homotopy fiber of $A_{0,i}'\to A|_{X_i'}$ (i.e., the shift of the cone). Similar results hold after base change to any further $X'_j$ or $X'$. But the cone of $A_{0,i}'\to A|_{X_i'}$ is itself a finitary sheaf (namely $A_{0,i}'\dotimes \mathbb Q_\ell/\mathbb Z_\ell$, which is a complex of \'etale sheaves), and $A_{0,i}$ is constructible. In particular, if the map from $A_{0,i}|_{X'}$ to the cone of $A_{0,i}'|_{X'}\to A|_{X'}$ vanishes, then this happens already over some $X'_j$; and then the maps form a torsor under
\[
\Hom_{\mathcal D_\cons(X',\mathcal O_L)}(A_{0,i}|_{X'},[A_{0,i}'|_{X'}\to A|_{X'}]) = \varinjlim_{j\geq i} \Hom_{\mathcal D_\cons(X'_j,\mathcal O_L)}(A_{0,i}|_{X'_j},[A_{0,i}'|_{X'_j}\to A|_{X'_j}]),
\]
using finitaryness of the homotopy fiber, and constructibility of $A_{0,i}$.

It remains to prove essential surjectivity, so assume given $A_0\in \mathcal D_\cons(X',\mathcal O_L)$ with $A_0[\tfrac 1\ell]\cong A|_{X'}\cong A_1|_{X'}[\tfrac 1\ell]$. Multiplying by a power of $\ell$ if necessary, we can assume that the map $A_0\to A|_{X'}$ arises from a map $A_0\to A_1|_{X'}$. The cone $B$ of $A_0\to A_1|_{X'}$ is then in $\mathcal D_\cons(X',\mathcal O_L)$ and killed by some power of $\ell$, so lies in the $\infty$-category from setting (B). As such, $B$ arises via pullback from some $B_i\in \mathcal D_\cons(X'_i,\mathcal O_L)$ and also the map $A_1 |_{X'}\to B$ can be approximated by a map $A_1 |_{X_i'}\to B_i$ (after increasing $i$). Then the homotopy fibre $A_{0,i}$ of $A_1|_{X_i'}\to B_i$ gives the desired approximation of $A_0$ over $X_i'$.
\end{proof}

Now for effectivity of descent, consider some arc-cover $Y\to X$ and some $A\in \mathcal D_\cons(Y,L)$ equipped with descent data. Let $\tilde{Y}$ be the finitary arc-sheaf of anima parametrizing $\mathcal O_L$-lattices in $A$ as in the proposition. The descent data for $A$ induce descent data for $\tilde{Y}$, which thus descends to a finitary arc-sheaf of anima $\tilde{X}$ over $X$ (necessarily arc-surjective over $X$, as $\tilde{Y}\to Y\to X$ are arc-covers). Moreover, by the case of $\mathcal O_L$-coefficients already handled, the universal $A_0\in \mathcal D_\cons(-,\mathcal O_L)$ over $\tilde{Y}$ descends to $\tilde{X}$. These reductions mean that we only need to prove descent along $\tilde{X}\to X$. This is an arc-cover, but $\tilde{X}$ is a finitary arc-sheaf. This means that there is some finitely presented $X$-scheme $X'\to X$ that is also an arc-cover, and a section of $\tilde{X}\to X$ over $X'$. In other words, we can reduce to the case of descent along a finitely presented arc-cover.

By the fully faithfulness already proved, we are also free to pass to a stratification. But any finitely presented arc-cover can, up to universal homeomorphisms, be refined by finite \'etale covers over a constructible stratification -- this is clear at points, and then follows by a spreading out argument. In other words, one can reduce to the case that $Y\to X$ is finite \'etale, and then even a $G$-torsor for some finite group $G$. In that case, the descent of $A$ is given by $(f_\ast A)^G$ (using that $|G|$ is invertible in $L$).
\end{proof}

As promised above, the proof gives the following corollary.

\begin{corollary}\label{cor:integralstructureexists} For any algebraic extension $L$ of $\mathbb Q_\ell$, the fully faithful functor
\[
\mathcal D_\cons(X,\mathcal O_L)\otimes_{\mathcal O_L} L\to \mathcal D_\cons(X,L)
\]
is an equivalence.
\end{corollary}

\begin{proof} We can assume that $L$ is finite over $\mathbb Q_\ell$. Take any $A\in \mathcal D_\cons(X,L)$. We note that the image of the functor is stable under cones and shifts. By Proposition~\ref{prop:constructiblelattices}, there is some \'etale cover $X'\to X$ over which an integral structure exists. Passing to a constructible stratification of $X$, we can assume that $X'\to X$ is finite \'etale and that $A$ is dualizable. Moreover, by the finitaryness aspect of Proposition~\ref{prop:constructiblelattices}, we can assume that $X$ is connected (as then any integral structure over a connected component spreads to an open and closed neighborhood). We can also assume that $X'$ is connected. We claim that all truncations of $A$ are still dualizable. This can be checked after pullback to the universal pro-finite \'etale cover $\tilde{X}\to X'\to X$, where $A$ becomes constant (using the integral structure over $X'$, this can be checked modulo powers of $\ell$, where everything reduces to usual finite \'etale local systems), and hence all truncations of $A$ are still constant sheaves on finitely dimensional $L$-vector spaces, and thus dualizable. Thus, we can assume that $A$ is concentrated in degree $0$. Then $A|_{\tilde{X}}$ is the constant sheaf on a finite-dimensional $L$-vector space $V$, and the descent data to $X$ is given by a continuous representation $\pi_1(X)\to \mathrm{GL}_L(V)$. Any such representation admits an invariant $\mathcal O_L$-lattice, finishing the proof.
\end{proof}

\begin{remark}\label{rem:largecategoryC} It is occasionally helpful to embed also the categories in setting (C) into larger categories that admit internal Hom's, direct images, and exceptional inverse images. This can be done: Assume first that $L$ is a finite extension of $\mathbb Q_\ell$. In that case, one can embed $\mathcal D_\cons(X,\mathcal O_L)$ into $\varprojlim_n \mathcal D_\et(X,\mathcal O_L/\ell^n)$, and this admits all six operations (via passage to limits). Now for general $L$ we can take $\varinjlim_{L'\subset L} \mathcal D_\et(X,\mathcal O_{L'})$ as $L'$ runs over finite subextensions of $L$, and with $L$-coefficients, we can formally invert $\ell$ and take the idempotent completion.
\end{remark}

\section{Universal local acyclicity}

In this section, we discuss universal local acyclicity, essentially following the approach of Lu--Zheng \cite{LuZhengULA}, but with a small shift in perspective. 

Fix any qcqs base scheme $S$ in which $\ell$ is invertible, and work in one of the settings (A), (B), or (C); in particular, we have fixed some coefficient $\mathbb Z_\ell$-algebra $\Lambda$, and abbreviate $D(X):=D(X,\Lambda)$ (where $D(X,\Lambda)$ is either $D_\et(X,\Lambda)$ or $D_\cons(X,\Lambda)$). We define a symmetric monoidal $2$-category $\mathcal C_S$ as follows. Its objects are schemes $f: X\to S$ separated and of finite presentation over $S$. The category of morphisms $\mathrm{Fun}_{\mathcal C_S}(X,Y)$ is given by $D(X\times_S Y)$; and composition is given by convolution, i.e.
\[
D(X\times_S Y)\times D(Y\times_S Z)\to D(X\times_S Z): (A,B)\mapsto A\star B = R\pi_{XZ!}(\pi_{XY}^\ast A\dotimes_\Lambda \pi_{YZ}^\ast B)
\]
where $\pi_{XY},\pi_{XZ},\pi_{YZ}$ are the obvious projections defined on $X\times_S Y\times_S Z$. The base change formula ensures that this gives an associative composition law. The identities are given by $R\Delta_{X/S!} \Lambda = R\Delta_{X/S\ast} \Lambda$, where $\Delta_{X/S}: X\to X\times_S X$ is the diagonal, which is a finitely presented closed immersion.

The symmetric monoidal structure on $\mathcal C_S$ is given on objects by $X\boxtimes Y = X\times_S Y$, and similarly on morphisms by exterior tensor products. Any object of $\mathcal C_S$ is dualizable: The dual of $X$ is given by $X$ itself, with unit $S\to X\times_S X$ and counit $X\times_S X\to S$ both given by $R\Delta_{X/S!} \Lambda\in D(X\times_S X)$. In particular, there are internal Hom's, and the internal Hom from $X$ to $Y$ is $X\times_S Y$.

We note that $\mathcal C_S$ is naturally isomorphic to the opposite $2$-category $\mathcal C_S^{\mathrm{op}}$ which exchanges the directions of the $1$-morphisms (but not of the $2$-morphisms), as $D(X\times_S Y)$ is naturally symmetric in $X$ and $Y$.

In \cite{FarguesScholze}, $\mathcal C_S$ was considered as a bare $2$-category, and the notion of adjoint maps in $2$-categories was employed to characterize universal local acyclicity. This could be done here again. However, we prefer to follow more closely \cite{LuZhengULA}. Indeed, we can also consider the (co)lax (co)slice $2$-category $\mathcal C'_S = _{S\backslash} \mathcal C_S$, which inherits a symmetric monoidal structure. Its objects are given by pairs $(X,A)$ where $f: X\to S$ is separated of finite presentation as before, and $A\in D(X) = \mathrm{Fun}_{\mathcal C_S}(S,X)$. A morphism $g: (X,A)\to (Y,B)$ in $\mathcal C'_S$ is given by some $C\in D(X\times_S Y)=\mathrm{Fun}_{\mathcal C_S}(X,Y)$ together with a map
\[
R\pi_{Y!}(\pi_X^\ast A\dotimes_\Lambda C)\to B,
\]
where $\pi_X,\pi_Y$ are the natural projections on $X\times_S Y$.

Then in setting (A), the symmetric monoidal $2$-category of cohomological correspondences (as in \cite{LuZhengULA}) has a natural symmetric monoidal functor to $\mathcal C'_S$, induced by sending a correspondence $c: C\to X\times_S Y$ to $Rc_! \Lambda\in D(X\times_S Y)$. Moreover, in setting (A), there are internal Hom's in $\mathcal C'_S$, where the internal Hom from $(X,A)$ to $(Y,B)$ is given by $(X\times_S Y,R\sHom(\pi_X^\ast A,R\pi_Y^!B))$. In fact, this already defines an internal Hom on the symmetric monoidal $2$-category considered by Lu--Zheng. This implies that $(X,A)$ is dualizable in Lu--Zheng's symmetric monoidal $2$-category if and only if it is dualizable in $\mathcal C'_S$ -- dualizability is then equivalent to the map $V\otimes \sHom(V,1)\to \sHom(V,V)$ being an isomorphism.

From setting (B) to setting (A), there is a fully faithful symmetric monoidal functor, which in particular preserves dualizable objects. We will see that all dualizable objects $(X,A)$ in fact lie in the essential image of (B) and are dualizable as objects in there, so these settings give rise to the same dualizable objects. For setting (C), we will develop techniques to reduce to setting (B).

Here is a general proposition that explains the relation between the approaches of \cite{LuZhengULA} and \cite{FarguesScholze}.

\begin{proposition}\label{prop:dualizablevsadjoint} Let $\mathcal C$ be a symmetric monoidal $2$-category with tensor unit $1\in \mathcal C$, and assume that all objects of $\mathcal C$ are dualizable. Let $\mathcal C'= _{1\backslash} \mathcal C$ be the lax coslice, which is itself a symmetric monoidal $2$-category. Then a morphism $f: 1\to X$ in $\mathcal C$ is a right adjoint if and only if $(X,f)\in \mathcal C'$ is dualizable.
\end{proposition}

\begin{proof} Assume that $(X,f)\in \mathcal C'$ is dualizable. As the forgetful functor $\mathcal C'\to \mathcal C$ is symmetric monoidal, its dual is of the form $(X^\ast,g^\ast)$ where $X^\ast\in \mathcal C$ is the dual of $X$, and $g^\ast: 1\to X^\ast$ is some map. The map $g^\ast$ is equivalent to a map $g: X\to 1$ by dualizability of $X$. We claim that $g$ is a left adjoint of $f$. To see this, we have to produce $2$-morphisms $\alpha: \mathrm{id}_1\to fg$ and $\beta: gf\to \mathrm{id}_X$ such that the composites
\[
f\xrightarrow{\alpha f} fgf\xrightarrow{f\beta} f,\ g\xrightarrow{g\alpha} gfg\xrightarrow{\beta f} g
\]
are the identity. But the dualizability of $(X,f)$ gives unit and counit maps
\[
(1,\mathrm{id}_1)\to (X\otimes X^\ast,f\otimes g^\ast), (X\otimes X^\ast,f\otimes g^\ast)\to (1,\mathrm{id}_1)
\]
satisfying similar conditions. The first map necessarily lies over the unit map $1\to X\otimes X^\ast$, and is then given by a $2$-morphism from the unit map $1\to X\otimes X^\ast$ to $f\otimes g^\ast: 1\to X\otimes X^\ast$. By dualizability of $X$, this is equivalent to a map from the identity on $X$ to $gf$. A similar analysis applies to the second map. Unraveling all the structures then shows that $g$ is a left adjoint of $f$. For the converse direction, one reverses all the steps.
\end{proof}

\begin{definition}\label{def:ULA} Let $f: X\to S$ be a separated map of finite presentation and $A\in D(X)$. Then $A$ is universally locally acyclic if
\[
A\in D(X)= \mathrm{Fun}_{\mathcal C_S}(S,X)
\]
is a right adjoint in $\mathcal C_S$; equivalently, if $(X,A)\in \mathcal C'_S$ is dualizable.
\end{definition}

We note that by the existence of internal Hom's in $\mathcal C'_S$ in setting (A), we get the following characterization.

\begin{proposition}\label{prop:ULAfirstchar} Let $f: X\to S$ be a separated map of finite presentation and $A\in D(X)$. Assume setting (A). Then $A$ is $f$-universally locally acyclic if and only if the map
\[
\pi_1^\ast \mathbb D_{X/S}(A)\dotimes_\Lambda \pi_2^\ast A\to R\sHom_\Lambda(\pi_1^\ast A,R\pi_2^! A)
\]
is an isomorphism in $D(X\times_S X)$, where $\mathbb D_{X/S}(A)$ denotes the relative Verdier dual, and $\pi_i: X\times_S X\to X$ the two projections.
\end{proposition}

\begin{proof} Indeed, an object $Y$ in a symmetric monoidal ($2$-)category with internal Hom's is dualizable if and only if the map $Y\otimes \sHom(Y,1)\to \sHom(Y,Y)$ is an isomorphism. Unraveling, we get this condition.
\end{proof}

In particular, using this proposition one verifies in setting (A) some basic properties of universal local acyclicity, such as that if $h: Y\to X$ is a map of separated $S$-schemes of finite presentation, then $Rh_\ast$ preserves universally locally acyclic sheaves if $h$ is proper, and $h^\ast$ preserves universally locally acyclic sheaves if $h$ is smooth. Also, if $g: S\to S'$ is a smooth map and $A$ is $f$-universally locally acyclic for some $f: X\to S$ as above, then $A$ is also $g\circ f$-universally locally acyclic. To see this in setting (A), use that $A$ being $f$-universally locally acyclic implies that for any $g: Y\to S$ separated of finite presentation and $B\in D(Y)$, the map
\[
\pi_1^\ast \mathbb D_{X/S}(A)\dotimes_\Lambda \pi_2^\ast B\to R\sHom_\Lambda(\pi_1^\ast A,R\pi_2^! B)
\]
om $D(X\times_S Y)$ is an isomorphism (as follows from dualizability of $A$). Now apply this to $Y=X\times_{S'} S$ and $B$ the pullback of $A$ to verify the condition of Proposition~\ref{prop:ULAfirstchar} for $A$ being $g\circ f$-universally locally acyclic.

We will only use these properties in setting (A), and only in the proof of Theorem~\ref{thm:nearbycycles}. However, these results also hold in settings (B) and (C): Indeed, universal local acyclicity in settings (A) and (B) is the same notion, while setting (C) reduces to setting (B) at least v-locally on $S$, using the integral structures of Proposition~\ref{prop:ULAintegralstructure}.

With this definition, one can prove the following properties. Here in setting (B) and (C) we denote by
\[
\mathbb D_{X/S}(A) = R\sHom_{D(X_\proet,\Lambda)}(A,Rf^! \Lambda)\in D(X_\proet,\Lambda)
\]
the internal Hom in $X_\proet$, where $Rf^! \Lambda$ comes from setting (A) in setting (B), and in setting (C) is defined via limits from setting (B).

\begin{proposition}\label{prop:basicpropertiesULA} Let $f: X\to S$ be a separated map of finite presentation and $A\in D(X)$ be $f$-universally locally acyclic.
\begin{enumerate}
\item[{\rm (i)}] Let $S'\to S$ be any map of schemes, and $f': X'=X\times_S S'\to S'$ the base change of $f$, and $A'\in D(X')$ the preimage of $A$. Then $A'$ is $f'$-universally locally acyclic.
\item[{\rm (ii)}] The relative Verdier dual $\mathbb D_{X/S}(A)=R\sHom_\Lambda(A,Rf^! \Lambda)$ of $A$ lies in $D(X)\subset D(X_\proet,\Lambda)$ and is $f$-universally locally acylic, and $(X,\mathbb D_{X/S}(A))$ is the dual of $(X,A)$ in $\mathcal C'_S$. In particular, the biduality map
\[
A\to \mathbb D_{X/S}(\mathbb D_{X/S}(A))
\]
is an isomorphism, and the formation of $\mathbb D_{X/S}(A)$ commutes with any base change in $S$.
\item[{\rm (iii)}] In setting (A), the complex $A$ is perfect-constructible.
\item[{\rm (iv)}] In setting (A), for any $(Y,B)\in \mathcal C'_S$, the map
\[
\pi_X^\ast A\dotimes_\Lambda \pi_Y^\ast B\to R\sHom_\Lambda(\pi_X^\ast \mathbb D_{X/S}(A),R\pi_Y^! B)
\]
is an isomorphism.
\item[{\rm (v)}] For any geometric point $\overline{x}\to X$ with image $\overline{s}\to S$, and generization $\overline{t}$ of $\overline{s}$, the maps
\[
A_{\overline{x}} = R\Gamma(X_{\overline{x}},A)\to R\Gamma(X_{\overline{x}}\times_{S_{\overline{s}}} S_{\overline{t}},A)\to R\Gamma(X_{\overline{x}}\times_{S_{\overline{s}}} \overline{t},A)
\]
are isomorphisms.
\end{enumerate}

In particular, condition (v) holds after any base change, so $A$ is universally locally acyclic in the usual sense.
\end{proposition}

We note that in many of the proofs, the case of setting (C) with rational coefficients is the hardest case. The reader is advised to omit that case on first reading; in particular, this is required to avoid any apparent vicious circles.

\begin{proof} Part (i) is a consequence of the observation that the pullback functors $\mathcal C_S\to \mathcal C_{S'}: X\mapsto X\times_S S'$ (and the induced functor $\mathcal C_S'\to \mathcal C_{S'}'$) are symmetric monoidal, and symmetric monoidal functors preserve dualizable objects. In setting (A), part (ii) follows from the description of internal Hom's in $\mathcal C_S'$. This formally gives the result in setting (B) as well, and in setting (C) for integral coefficients by reducing to a finite extension of $\mathbb Q_\ell$ and then via limits to setting (B). Setting (C) with rational coefficients is addressed later.

For part (iii), note that by Theorem~\ref{thm:arcdescentD} we can argue v-locally on $S$, so we can assume that every connected component of $S$ is the spectrum of an absolutely integrally closed valuation ring. In that case $X$ has finite $\ell$-cohomological dimension by Lemma~\ref{lem:finitecohomdim} below, and so perfect-constructibility is equivalent to compactness in $D_\et(X,\Lambda)$. But by dualizability of $A$, the map
\[
\pi_X^\ast \mathbb D_{X/S}(A)\dotimes_\Lambda \pi_Y^\ast B\to R\sHom_\Lambda(\pi_X^\ast A,R\pi_Y^! B)
\]
is an isomorphism for any $(Y,B)$; in particular, applying this in case $Y=X$ and taking $R\Delta_{X/S}^!$, we find
\[
R\Delta_{X/S}^!(\pi_1^\ast \mathbb D_{X/S}(A)\dotimes_\Lambda \pi_2^\ast B)\cong R\sHom_\Lambda(A,B),
\]
and thus
\[
R\Hom_{D_\et(X,\Lambda)}(A,B)\cong R\Gamma(X,R\Delta_{X/S}^!(\pi_1^\ast \mathbb D_{X/S}(A)\dotimes_\Lambda \pi_2^\ast B)).
\]
Now the functor on the right commutes with all direct sums in $B$, and hence $A$ is compact, as desired.

Part (iv) follows from the first displayed formula in the previous paragraph, applied to $\mathbb D_{X/S}(A)$, using also (ii). For part (v) in setting (A), we first specialize part (iv) to $B=\Lambda$ for any separated $g: Y\to S$ of finite presentation, and apply $R\pi_{X\ast}$. Then the left-hand side becomes $R\pi_{X\ast} A|_{X\times_S Y}$, while the right-hand side becomes
\[
R\sHom_\Lambda(\mathbb D_{X/S}(A),R\pi_{X\ast} R\pi_Y^! \Lambda) = R\sHom_\Lambda(\mathbb D_{X/S}(A),Rf^! Rg_\ast \Lambda).
\]
Applying part (iv) again for $(S,Rg_\ast \Lambda)$, we see that the map
\[
A\dotimes_\Lambda Rg_\ast \Lambda\to R\sHom_\Lambda(\mathbb D_{X/S}(A),Rf^! Rg_\ast \Lambda)
\]
is also an isomorphism. In total, we see that the natural map
\[
A\dotimes_\Lambda Rg_\ast \Lambda\to R\pi_{X\ast} A|_{X\times_S Y}
\]
is an isomorphism. A priori, this holds for all separated $Y$ of finite presentation, but then by passage to limits it follows for all (qcqs) $S$-schemes $Y$. In particular, after base changing to $S_{\overline{s}}$, we can apply it to $Y=S_{\overline{t}}$ or $Y=\overline{t}$. Taking stalks of this isomorphism at geometric points $\overline{x}\to X$ over $\overline{s}\to S$ then proves (v) in setting (A). This formally gives the result also in setting (B), and in setting (C) for integral coefficients by passage to limits.

It remains to prove parts (ii) and (v) in setting (C) with rational coefficients. Note first that the result is automatic if $A\in \mathcal D_\cons(X,L)$ is of the form $A_0[\tfrac 1\ell]$ where $A_0\in \mathcal D_\cons(X,\mathcal O_L)$ is universally locally acyclic. In general, Proposition~\ref{prop:ULAintegralstructure} below ensures that this happens arc-locally on $S$. Part (ii) then follows in general by arc-descent. More precisely, choose an arc-cover $S_0 \to S$ over which $A$ admits a ULA integral structure. Let $S_\bullet \to S$ be the Cech nerve, and let $f_n: X_n \to S_n$ be the evident base change. Then $\mathbb D_{X_0/S_0}(A|X_0)$ is constructible and commutes with any base change on $S_0$, so $(\mathbb D_{X_n/S_n}(A|X_n) )_n$ defines an object of $D_\cons^{\mathrm{cart}}(X_\bullet,L) \simeq D_\cons(X,L)$ which computes $\mathbb D_{X/S}(A)$. But then all $\mathbb D_{X_n/S_n}(A|X_n)$ are constructible and commute with any additional base change $S' \to S$, so $\mathbb D_{X/S}(A)$ is constructible and commutes with any additional base change on $S$. Biduality is then immediate, since formation of the biduality map commutes with pullback to $X_0$, where we know the result.

Part (v) is slightly trickier, as the statement in itself is not amenable to arc-descent. Note first that in part (v) it is enough to prove the first isomorphism; the composite isomorphism is just its variant after base change to the closure of $\overline{t}$ in $S_{\overline{s}}$. We replace $S_{\overline{t}}\to S_{\overline{s}}$ by any pro-\'etale map $g: T\to S$. We can then ask whether the map
\[
A\widehat{\dotimes} Rg_\ast \mathbb Z_\ell\to R\tilde{g}_\ast A|_{X\times_S T}
\]
is an isomorphism, where $\tilde{g}: X\times_S T\to X$ is the base change of $g$, and $\widehat{\dotimes}$ denotes the $\ell$-adically completed tensor product, using any integral structure on $A$ (which exists by Corollary~\ref{cor:integralstructureexists}) -- the resulting $\ell$-adically completed tensor product is independent of the choice of integral structure. This statement holds true when $A$ admits a universally locally acyclic integral structure (by the proof of (v)), and hence holds true over a proper cover $h: S'\to S$ by Proposition~\ref{prop:ULAintegralstructure}, or in fact for all the terms in the induced Cech nerve $S'^{\times_S n}$. Applying proper pushforward along $S'^{\times_S n}\to S$ and totalizing, we get the desired isomorphism.
\end{proof}

We used the following result on finite cohomological dimension due to Gabber \cite{GabberOberwolfach2020}.

\begin{lemma}\label{lem:finitecohomdim} Let $S$ be an affine scheme over $\mathbb Z[\tfrac 1\ell]$ all of whose connected components are spectra of absolutely integrally closed valuation rings, and let $f: X\to S$ be an affine scheme of finite type. Let $d$ be the maximal fibre dimension of $f$. Then the $\ell$-cohomological dimension of $X$ is bounded by $d+1$.
\end{lemma}

In fact, Gabber showed that one can bound the $\ell$-cohomological dimension by $d$, by proving an even more general relative version of Artin vanishing. We will recall his result in Proposition~\ref{prop:artinvanishing} below.

\begin{proof} As $\pi_0 S$ is profinite, it suffices to check this on connected components. We can also reduce to the case that $S=\Spec V$ where $V$ is of finite rank, and to sheaves $\mathcal F$ concentrated in one fibre. Base changing to the closure of this fibre, we can assume that this is the generic fibre of $S$. Let $S'\subset S$ be the open subset consisting of the generic point $\eta$ and its immediate specialization (if it exists). By arc-excision applied to the cover of $S$ by $S'$ and $S\setminus \{\eta\}$, we find that $R\Gamma(X,\mathcal F)=R\Gamma(X\times_S S',\mathcal F)$; so we can assume that $S$ is of rank (at most) $1$. The case of fields is given by Artin vanishing. Now let $\tilde{X}$ be the henselization of $X$ at the special fibre. Then there is a triangle
\[
R\Gamma(X,\mathcal F)\to R\Gamma(X_\eta,\mathcal F)\oplus R\Gamma(\tilde{X},\mathcal F)\to R\Gamma(\tilde{X}_\eta,\mathcal F),
\]
and by Gabber's affine analogue of proper base change \cite{GabberAffine}, $R\Gamma(\tilde{X},\mathcal F)=0$ (as we assumed that $\mathcal F$ is concentrated on the generic fibre). But by Artin vanishing, both $R\Gamma(X_\eta,\mathcal F)$ and $R\Gamma(\tilde{X}_\eta,\mathcal F)$ sit in degrees $\leq d$, giving the claim.
\end{proof}

\begin{remark}\label{rem:rungevanishing} The proof shows that the vanishing in cohomological degree $d+1$ has the following reinterpretation in terms of rigid-analytic geometry. Let $V$ be a complete rank $1$ valuation ring with algebraically closed fraction field $K$, and let $X$ be an affine scheme of finite type over $V$, of relative dimension $d$. Let $\hat{X}/\mathrm{Spf}\, V$ be its completion, and let $\hat{X}_K$ be its generic fibre as a rigid-analytic variety; this is an open affinoid subset of the rigid-analytic variety associated to $X_K$. Finally, let $\mathcal F$ be any constructible sheaf, of torsion order invertible in $V$. Then the map
\[
H^d(X_K,\mathcal F)\to H^d(\hat{X}_K,\mathcal F)
\]
is surjective. This is a rigid-analytic analogue (with constructible coefficients) of a known property of Runge pairs in complex-analytic geometry, cf.~e.g.~\cite{AndreottiNarasimhan}. (We thank Mohan Ramachandran for making us aware of this reference.)
\end{remark}

Next, we analyze arc-descent properties.

\begin{proposition}\label{prop:ULAfinitaryarc} Let $f: X\to S$ be a separated map of finite presentation. Consider the functor taking any $S'$ over $S$ to the $\infty$-category $\mathcal D^{\mathrm{ULA}}(X'/S')\subset \mathcal D(X')$ of universally locally acyclic sheaves on $X'=X\times_S S'$ over $S'$. This defines an arc-sheaf of $\infty$-categories, which is finitary in settings (A) and (B).

In particular, if $A\in D(X)$ and $S'\to S$ is an arc-cover such that $A|_{X'}$ is universally locally acyclic over $S'$, then $A$ is universally locally acyclic over $S$.
\end{proposition}

\begin{proof} As settings (A) and (B) give rise to the same notion of universally locally acyclic sheaves, we can assume that we are in setting (B) or (C). Then $\mathcal D^{\mathrm{ULA}}(X'/S')\subset \mathcal D(X')=\mathcal D_\cons(X',\Lambda)$, and we know that the latter is an arc-sheaf by Theorem~\ref{thm:arcdescentD}. Thus, we only need to prove effectivity of descent, which is exactly the final sentence, and that it is a finitary arc-sheaf in setting (B). Finitaryness in setting (B) follows from $\mathcal C'_S$ taking cofiltered limits of affine schemes $S$ to filtered colimits of symmetric monoidal $2$-categories (and hence the same happens on dualizable objects).

For the final sentence, note that the question whether the Verdier dual (formed as a pro-\'etale sheaf, as in Proposition~\ref{prop:basicpropertiesULA}~(ii)) is again in $\mathcal D_\cons(X,\Lambda)$ and commutes with base change in $S$ can be checked arc-locally on $S$. Thus, we have a well-defined dual $A^\vee = \mathbb D_{X/S}(A)$ of $A$. Similarly, one can produce the unit and counit maps via arc-descent. Alternatively, use the characterization of Proposition~\ref{prop:ULAfirstchar} in setting (A), which can be adapted to setting (C) by working with $\ell$-adically completed derived categories (resp.~the isogeny category).
\end{proof}

\begin{proposition}\label{prop:ULAintegralstructure} Let $f: X\to S$ be a separated map of finite presentation, and consider setting (C). Let $A\in \mathcal D^{\mathrm{ULA}}(X/S,L)$, and consider the functor taking $S'/S$ to the $\infty$-category of $A_0\in \mathcal D^{\mathrm{ULA}}(X'/S',\mathcal O_L)$ with an isomorphism $A_0[\tfrac 1\ell]\cong A|_{X'}$, where $X'=X\times_S S'$. This defines a finitary arc-sheaf of $\infty$-categories that admits a section over a finitely presented proper surjection $S'\to S$.
\end{proposition}

\begin{proof} By Proposition~\ref{prop:ULAfinitaryarc} and Theorem~\ref{thm:arcdescentD}, it is an arc-sheaf of $\infty$-categories. Moreover, note that $A_0\in \mathcal D_\cons(X'/S',\mathcal O_L)$ is universally locally acyclic if and only if $A_0/\ell\in \mathcal D_\cons(X'/S',\mathcal O_L/\ell)$ is; indeed, by approximation we may assume that $L$ is a finite extension of $\mathbb Q_\ell$, and then the condition lifts to $\mathcal O_L/\ell^n$, and then to $\mathcal O_L$ by passing to the limit. Then it follows from Proposition~\ref{prop:ULAfinitaryarc} and Proposition~\ref{prop:constructiblelattices} that it is a finitary arc-sheaf. Moreover, Theorem~\ref{thm:nearbycycles} implies that it satisfies the valuative criterion of properness, i.e.~for any absolutely integrally closed valuation ring $V$ over $S$ with fraction field $K$, the value at $V$ maps isomorphically to the value at $K$; moreover, that theorem shows that the value at $K$ is nonempty. Now the result follows from Lemma~\ref{lem:propersections}.
\end{proof}

\begin{lemma}\label{lem:propersections} Let $S$ be a qcqs scheme and let $\mathcal F$ be a finitary Zariski sheaf of anima satisfying the valuative criterion of properness. Assume that for any $S'\to S$ and any two sections $a,b\in\mathcal F(S')$, the sheaf of isomorphisms between $a$ and $b$ is $n$-truncated for some $n$, and that for any algebraically closed field $K$ over $S$, $\mathcal F(\mathrm{Spec} K)$ is nonempty. Then there is a finitely presented proper surjection $S'\to S$ with $\mathcal F(S')$ nonempty.
\end{lemma}

\begin{proof} By an approximation argument (using that $\mathcal F$ is finitary), we can reduce to the case that $S$ is irreducible; let $K$ be its fraction field, and fix an algebraic closure $\overline{K}$ of $K$. Consider the cofiltered category of finitely presented proper $S$-schemes $S'$ with a fixed $\overline{K}$-point; we claim that the restriction map
\[
\mathrm{colim}_{S'} \mathcal F(S')\to \mathcal F(\overline{K})
\]
is an isomorphism, which implies the lemma.

First, we prove this claim for $n$-truncated $\mathcal F$ by induction on $n$. For $n=-2$, there is nothing to prove as then $\mathcal F=\ast$. Now we prove that the map is an injection, i.e.~whenever $a,b\in\mathcal F(S')$ for some proper $S'$ over $S$, then
\[
\mathrm{colim}_{S''} \mathrm{Isom}_{\mathcal F(S'')}(a|_{S''},b|_{S''})\to \mathrm{Isom}_{\mathcal F(\overline{K})}(a|_{\overline{K}},b|_{\overline{K}})
\]
is an isomorphism. Replacing $S$ by $S'$ and $\mathcal F$ by the sheaf of isomorphisms between $a$ and $b$ gives an $n-1$-truncated sheaf, reducing to the induction hypotheses.

Now take any section of $s\in \mathcal F(\overline{K})$. By quasicompactness of the Zariski--Riemann space $\mathrm{lim}_{S'} |S'|$ and the valuative criterion of properness (and finitaryness), there is some proper $S'$ that admits a cover by finitely many open subsets $U_i\subset S'$ for which $s$ lies in the image of $\mathcal F(U_i)\to \mathcal F(\overline{K})$. We can replace $S$ by $S'$, and so assume that $s$ is locally in the image of $\mathcal F(S)\to \mathcal F(\overline{K})$. It remains to glue the local sections, but this is possible (after passing to some proper cover) by the claim on isomorphisms.

This proves the result for $n$-truncated $\mathcal F$. But now the argument showing that the map is an injection applies in general, as does the last paragraph.
\end{proof}

In fact, one can check universal local acyclicity after pullback to absolutely integrally closed, rank $1$ valuation rings.

\begin{corollary}\label{cor:ULAtestrank1} Let $f: X\to S$ be a separated map of finite presentation and $A\in D(X)$ in setting (B) or (C). Then $A$ is $f$-universally locally acyclic if and only if for all rank $1$ valuation rings $V$ with algebraically closed fraction field $K$ and all maps $\Spec V\to S$, the restriction $A|_{X_V}\in D(X_V)$ to $X_V=X\times_S \Spec V$ is universally locally acyclic over $V$. 
\end{corollary}

\begin{proof} In setting (B), this is a consequence of $S'\mapsto \mathcal D^{\mathrm{ULA}}(X'/S',\Lambda)\subset \mathcal D_\cons(X',\Lambda)$ being a finitary arc-sheaf: We may first assume that all connected components of $S$ are spectra of absolutely integrally closed valuation rings, and then by finitaryness we can assume that $S$ is the spectrum of an absolutely integrally closed valuation ring, in fact one of finite rank. Then by arc-descent one can reduce to the rank $1$ case, as desired.

In setting (C) with integral coefficients, the result follows formally from setting (B). With rational coefficients, consider the finitary arc-sheaf of anima parametrizing universally locally acyclic $A_0$ with integral coefficients and $A_0[\tfrac 1\ell]\cong A$, as in Proposition~\ref{prop:ULAintegralstructure}. It suffices to see that this admits a section over an arc-cover of $S$. By finitaryness, we can reduce to the case that $S$ is the spectrum of an absolutely integrally closed valuation ring. By Theorem~\ref{thm:nearbycycles}, there is a unique universally locally acyclic extension of the restriction to the generic fibre $j: X_\eta\hookrightarrow X$. Replacing $A$ by the cone of $A\to Rj_\ast j^\ast A$, we can assume that the restriction of $A$ to $X_{\eta}$ is trivial. As $A$ is constructible, the image of its support in $S$ is constructible; we can thus find a morphism $\Spec V\to S$ from a rank $1$ valuation ring whose closed point maps into the support of $A$, but whose generic point does not, and replace $S$ by this base change. But then $A$ is universally locally acyclic by assumption, and its restriction to the generic fibre vanishes, so $A=0$ by Theorem~\ref{thm:nearbycycles}.
\end{proof}

\section{Nearby cycles}

The following theorem is essentially due to Lu--Zheng, \cite[Section 3]{LuZhengULA}.

\begin{theorem}\label{thm:nearbycycles} Let $S=\Spec V$ be an absolutely integrally closed valuation ring $V$ with fraction field $K$. Let $X$ be a separated scheme of finite presentation over $S$, with generic fibre $X_\eta$. Consider one of the settings (B) and (C).

The restriction functor
\[
D^{\mathrm{ULA}}(X/S)\to D(X_\eta)
\]
is an equivalence, whose inverse is given by $Rj_\ast: D(X_\eta)\subset D(X_{\eta,\proet},\Lambda)\to D(X_\proet,\Lambda)$ for $j: X_\eta\to X$ the inclusion.
\end{theorem}

Before proving the theorem, we note a couple of consequences.

\begin{corollary}\label{cor:nearbycyclesnice} In the situation of Theorem~\ref{thm:nearbycycles}, the functor
\[
Rj_\ast: D(X_\eta)\subset D(X_{\eta,\proet},\Lambda)\to D(X_\proet,\Lambda)
\]
has the following properties:
\begin{enumerate}
\item[{\rm (i)}] its image is contained in $D(X)=D_\cons(X,\Lambda)$;
\item[{\rm (ii)}] its formation commutes with any pullback along a map $S'=\Spec V'\to \Spec V$ where $V\to V'$ is a flat map of absolutely integrally closed valuation rings;
\item[{\rm (iii)}] it commutes with (relative) Verdier duality;
\item[{\rm (iv)}] it satisfies a K\"unneth formula: if $Y$ is another scheme of finite presentation over $S$, then the diagram
\[\xymatrix{
D(X_\eta)\times D(Y_\eta)\ar[r]^{\boxtimes}\ar[d]^{Rj_\ast \times Rj_\ast} & D((X\times_S Y)_\eta)\ar[d]^{Rj_\ast}\\
D(X)\times D(Y)\ar[r]^{\boxtimes} & D(X\times_S Y)
}\]
commutes.
\end{enumerate}

Passing to the closed fibre $i: X_s\to X$, the nearby cycles functor
\[
R\psi = i^\ast Rj_\ast: D(X_\eta)\to D(X_s)
\]
has the same properties (assuming that $V\to V'$ is faithfully flat in (i)).
\end{corollary}

We note that part (iii) was observed by Fujiwara, cf. \cite[Proof of Lemma 1.5.1]{Fujiwara}.

\begin{proof} Part (i) is part of Theorem~\ref{thm:nearbycycles}. Part (ii) follows from preservation of universal local acyclicity under pullback. Part (iii) follows from preservation of universal local acyclicity under relative Verdier duality. We note that to get the same result for $R\psi= i^\ast Rj_\ast$ we also use that formation of relative Verdier duals commutes with any base change for universally locally acyclic sheaves. Finally, part (iv) follows from preservation of universal local acyclicity under exterior tensor products.
\end{proof}

Now we prove Theorem~\ref{thm:nearbycycles}.

\begin{proof}[Proof of Theorem~\ref{thm:nearbycycles}] First we assume that we are in setting (B), which we may embed into setting (A). We start by proving fully faithfulness. In fact, for any $A\in D^{\mathrm{ULA}}(X/S)$, the natural map
\[
A\to Rj_\ast j^\ast A
\]
must be an isomorphism. This follows from Proposition~\ref{prop:basicpropertiesULA}~(iv) applied to $Y=\Spec K\to \Spec V$. (We note that for a general valuation ring, this may not be of finite type over $S$, but one can still write it as a limit of quasicompact open subsets, giving the conclusion by passing to filtered colimits.) This immediately gives fully faithfulness. It remains to show that
\[
j^\ast: D^{\mathrm{ULA}}(X/S)\hookrightarrow D_\cons(X_\eta)
\]
is essentially surjective: Indeed, we have just seen that the inverse functor is necessarily given by $Rj_\ast$. We note that even for $S=\Spec K$ a field, this is Deligne's theorem on universal local acyclicity over a field, which we will reprove here.

At this point, we follow an argument that goes back to Deligne's proof of constructibility of nearby cycles, \cite{SGA412}, cf.~also the appendix of \cite{IllusieAutour}.

We argue by induction on the (relative) dimension $d$ of $X$. We first prove that there is some closed subset $Z\subset X$ whose special fiber is finite such that $(Rj_\ast A)|_{X\setminus Z}$ is universally locally acyclic over $S$. To see this, we may assume that $X$ is affine, and pick some map $g: X\to \mathbb A^1_S$. Taking the strict henselization of $\mathbb A^1_S$ at the generic point of the special fibre gives the spectrum $\Spec W$ of some (henselian) valuation ring $W$ over $V$. Its fraction field $L$ may not be algebraically closed, but at least its absolute Galois group is pro-$p$, where $p$ is the residue characteristic of $V$ (if positive; otherwise $L$ is indeed algebraically closed): Indeed, the residue field of $W$ is separably closed, and its value group agrees with the value group of $V$, which is divisible. Let $\overline{W}$ be an absolute integral closure of $W$, and $X'=X\times_{\mathbb A^1_S} \Spec \overline{W}$, with $j': X'_\eta\to X'$ the open immersion of the generic fibre. By induction, $Rj'_* (A|_{X'_\eta})$ is universally locally acyclic over $\Spec \overline{W}$. By approximation, we can replace $\overline{W}$ by a finite extension of $p$-power degree of $W$; by the structure of curves over absolutely integrally closed valuation rings (in particular, semistable reduction), any such finite extension is itself the strict henselization of a smooth curve $C$ over $V$ at a generic point of the special fibre. Thus, by a spreading out argument, we can construct an \'etale map $C_0\to \mathbb A^1_S$ and a finite extension of $p$-power degree $C\to C_0$ such that $C$ and $C_0$ are smooth curves over $V$, and such that $X_C=X\times_{\mathbb A^1_S} C$ has the property that $Rj_{X_{C,\eta}*}(A|_{X_{C,\eta}})$ is universally locally acyclic over $C$, and thus also over $S=\Spec V$ (as $C\to S$ is smooth). As $C\to C_0$ is finite of $p$-power degree, a trace argument implies that also $Rj_{X_{C_0,\eta}*}(A|_{X_{C_0,\eta}})$ is universally locally acyclic over $S$, where $X_{C_0} = X\times_{\mathbb A^1_S} C_0$. Now \'etale descent implies that, on the preimage in $X$ of the open image of $C_0\to \mathbb A^1_S$, also $Rj_* A$ is universally locally acyclic over $S$. The union of the open subsets of $X$ constructed this way define an open subset of $X$ whose complement must have finite special fibre (as otherwise there is some projection to $\mathbb A^1_S$ whose image contains the generic point).

The next reduction is to assume that $X$ is proper, noting that any $X$ admits a compactification (by Nagata, or simply locally by embedding into projective space); also, any $A\in D_\cons(X_\eta,\Lambda)$ extends to the compactification through extension by $0$. In this case, the closed subset $Z\subset X$ constructed above is itself proper over $\Spec V$, with finite special fiber, and thus itself finite over $\Spec V$. Now we are finished by Lemma~\ref{lem:ULAlocalglobal} below.

This finishes the proof in setting (B). Setting (C) formally reduces to the case of a finite extension $L/\mathbb Q_\ell$. In the case of $\mathcal O_L$-coefficients, one can then formally reduce to $\mathcal O_L/\ell^n$-coefficients, which is setting (B). In the setting of $L$-coefficients, we note that essential surjectivity follows from the case of $\mathcal O_L$-coefficients, and this also proves the claim that $Rj_\ast$ takes image in universally locally acyclic sheaves. It remains to prove that if $A\in D^{\mathrm{ULA}}(X/S)$, then the map $A\to Rj_\ast j^\ast A$ is an isomorphism. Noting that $Rj_\ast j^\ast A\in D^{\mathrm{ULA}}(X/S)$ by what we already proved, it suffices to prove that $A=0$ if $j^\ast A=0$. To see this, note that the support of $A$ is a constructible subset of $X$ and hence its image in $S$ is also constructible. Thus, its image in $S$ has a generic point; by base change, we can assume that this is the closed point of $S$. As then the closed point of $S$ is a constructible closed subset, its open complement is quasicompact and hence has a closed point, which we can assume is the generic point of $S$; we can thus assume that $V$ is of rank $1$. Now using Remark~\ref{rem:largecategoryC} one can define a variant of $\mathcal C_S$ using these big categories that admits internal Hom's, and this implies that $A=R\sHom_L(A^\vee,Rf^! L)$ is the Verdier dual (in the sense of the categories in Remark~\ref{rem:largecategoryC}) of its dual $A^\vee$ in $\mathcal C'_S$. But $Rf^! L = Rj_\ast j^\ast Rf^! L$, and hence $A=Rj_\ast R\sHom_L(j^\ast A^\vee,j^\ast Rf^! L)$ where $j^\ast A^\vee = (j^\ast A)^\vee = 0$, and hence $A=0$, as desired.
\end{proof}

\begin{lemma}\label{lem:ULAlocalglobal} Let $f: X\to S$ be a finitely presented proper map of qcqs schemes. Let $A\in D_\et(X,\Lambda)$ in setting (A), and assume that there is some closed subscheme $Z\subset X$ that is finite over $S$, such that $A|_{X\setminus Z}$ is universally locally acyclic. Moreover, assume that $Rf_\ast A\in D_\et(S,\Lambda)$ is universally locally acyclic over $S$, i.e.~locally constant with perfect fibres. Then $A$ is universally locally acyclic over $S$.
\end{lemma}

\begin{proof} We have to see that the map
\[
(X,A)^\vee\otimes (X,A)\to \sHom_{\mathcal C'_S}((X,A),(X,A))
\]
in $\mathcal C'_S$ is an isomorphism; equivalently, the map
\[
\pi_1^\ast \mathbb D_{X/S}(A)\dotimes_\Lambda \pi_2^\ast A\to R\sHom_\Lambda(\pi_1^\ast A,R\pi_2^! A)
\]
is an isomorphism on $X\times_S X$. We will prove that it is an isomorphism away from $Z\times_S Z$, and after taking the pushforward to $S$. This will give the claim: The cone of this map is supported on $Z\times_S Z$, which is finite over $S$, hence pushforward to $S$ is conservative.

Restricting to $(X\setminus Z)\times_S X$, the map is an isomorphism as $A|_{X\setminus Z}$ is universally locally acyclic over $S$, so that in $\mathcal C'_S$, we have
\[
(X\setminus Z,A|_{X\setminus Z})^\vee\otimes (X,A)\cong \sHom_{\mathcal C'_S}((X\setminus Z,A|_{X\setminus Z}),(X,A)).
\]
Similarly, the restriction to $X\times_S (X\setminus Z)$ is an isomorphism, using this time that
\[
(X,A)^\vee\otimes (X\setminus Z,A|_{X\setminus Z})\cong \sHom_{\mathcal C'_S}((X,A),(X\setminus Z,A|_{X\setminus Z})),
\]
by dualizability of the second factor.

It remains to prove that the pushforward to $S$ is an isomorphism. But unraveling, this exactly amounts to the question whether $Rf_\ast A$ is universally locally acyclic over $S$, which we have assumed.
\end{proof}

Using these results, we see that our definition of universal local acyclicity agrees with the usual definition. More precisely:

\begin{theorem}\label{thm:ULAcorrect} Let $f: X\to S$ be a separated map of finite presentation between qcqs schemes and let $A\in D(X)$ in one of the settings (B) and (C). The following conditions are equivalent.
\begin{enumerate}
\item[{\rm (i)}] The pair $(X,A)$ defines a dualizable object in the symmetric monoidal $2$-category of cohomological correspondences over $S$.
\item[{\rm (ii)}] The following condition holds after any base change in $S$. For any geometric point $\overline{x}\to X$ mapping to a geometric point $\overline{s}\to S$, and a generization $\overline{t}\to S$ of $\overline{s}$, the map
\[
A|_{\overline{x}} = R\Gamma(X_{\overline{x}},A)\to R\Gamma(X_{\overline{x}}\times_{S_{\overline{s}}} S_{\overline{t}},A)
\]
is an isomorphism.
\item[{\rm (iii)}] The following condition holds after any base change in $S$. For any geometric point $\overline{x}\to X$ mapping to a geometric point $\overline{s}\to S$, and a generization $\overline{t}\to S$ of $\overline{s}$, the map
\[
A|_{\overline{x}} = R\Gamma(X_{\overline{x}},A)\to R\Gamma(X_{\overline{x}}\times_{S_{\overline{s}}} \overline{t},A)
\]
is an isomorphism.
\item[{\rm (iv)}] After base change along $\Spec V\to S$ for any rank $1$ valuation ring $V$ with algebraically closed fraction field $K$ and any geometric point $\overline{x}\to X$ mapping to the special point of $\Spec V$, the map
\[
A|_{\overline{x}} = R\Gamma(X_{\overline{x}},A)\to R\Gamma(X_{\overline{x}}\times_{\Spec V}\Spec K,A)
\]
is an isomorphism.
\end{enumerate}
Moreover, these conditions are stable under any base change, and can be checked arc-locally on $S$.
\end{theorem}

\begin{proof} By Proposition~\ref{prop:basicpropertiesULA}, (i) implies (ii) and (iii), and each of them has (iv) as a special case. Thus, it remains to prove that (iv) implies (i). By Corollary~\ref{cor:ULAtestrank1}, we can assume that $S=\Spec V$ is the spectrum of an absolutely integrally closed valuation ring of rank $1$. Then Theorem~\ref{thm:nearbycycles} shows that (i) is equivalent to the map $A\to Rj_\ast j^\ast A$ being an isomorphism, where $Rj_\ast j^\ast A$ is also constructible. It is clearly an isomorphism in the generic fibre, so one has to check that it is an isomorphism in the special fibre. Checking stalkwise, this is exactly the condition (iv).

The final sentence comes from Proposition~\ref{prop:ULAfinitaryarc}.
\end{proof}

We note the following corollary that we will use in the next section; it states that invariance of cohomology under change of algebraically closed base field holds in fact more generally for change of absolutely integrally closed valuation rings.

\begin{corollary}[{\cite[Corollary 4.2.7]{HuberBook}}]\label{cor:invarianceofcohomology} Let $V\to V'$ be a faithfully flat map of absolutely integrally closed valuation rings and let $X$ be a scheme of finite type over $V$, with base change $X'$ over $V'$. Let $A\in D_\et(X,\Lambda)$ in setting (A). Then the map
\[
R\Gamma(X,A)\to R\Gamma(X',A|_{X'})
\]
is an isomorphism.

In case $A\in D_\et^+(X,\Lambda)$, the same statement holds for any scheme $X$ over $V$, not necessarily of finite type.
\end{corollary}

\begin{proof} We can assume $\Lambda=\mathbb Z/\ell^n\mathbb Z$, and we can assume that $X$ is affine, and of finite presentation (by choosing a closed immersion). As $X$ has finite $\ell$-cohomological dimension by Lemma~\ref{lem:finitecohomdim}, we can reduce to $A$ being constructible. By approximation, we can assume that $V$ is of finite rank. Arguing by induction on the rank of $V$, we can use Theorem~\ref{thm:nearbycycles} and the triangle $A\to Rj_\ast j^\ast A\to A'$ to reduce to the case that $A=Rj_\ast j^\ast A$ is universally locally acyclic (as $A'$ is supported on a proper closed subset of $\Spec V$, and we can apply the induction hypothesis). In that case $R\Gamma(X,A) = R\Gamma(X_K,A|_{X_K})$ where $K$ is the fraction field of $V$, and similarly for $V'$. This reduces us to the case where $V$ and $V'$ are algebraically closed fields, and the result is the classical result on invariance of cohomology under change of algebraically closed base field.

For the final sentence, we can reduce to $A$ sitting in a single degree and $\Lambda=\mathbb Z/\ell^n\mathbb Z$, and then again to constructible sheaves. Moreover, one can assume $X$ is affine. Now the result follows by writing $X$ as a cofiltered limit of affine schemes of finite type, approximating the constructible sheaf, and using that \'etale cohomology becomes a filtered colimit.
\end{proof}

\section{Universally submersive descent}

The results of this section are due to Gabber \cite{GabberLetterMathew}.

\begin{definition}\label{def:submersion} A qcqs map $f: Y\to X$ of schemes is a submersion if the map $|Y|\to |X|$ is a quotient map. The map $f: Y\to X$ is a universal submersion if any base change of $f$ is a submersion.
\end{definition}

\begin{proposition}\label{prop:submersion} A qcqs map $f: Y\to X$ is a universal submersion if and only if for all valuation rings $V$ with fraction field $K\supsetneq V$, and a map $\Spec V\to X$, the map $Y_K\to Y_V$ is not a closed immersion.

In particular, a universal submersion is an arc-cover.
\end{proposition}

\begin{proof} Assume that $f$ is a universal submersion. To check the condition, we may assume that $X=\Spec V$. Assume that $Y_K\to Y_V$ was a closed immersion. Then the preimage of $ \Spec K\subset \Spec V$ is closed, so by the assumption that $f$ is a universal submersion, also $ \Spec K \subset \Spec V$ is closed, which is a contradiction.

In the converse direction, as the condition is stable under base change, it suffices to show that $f$ is a submersion. Applying the condition to rank $1$ valuation rings $V=k[[t]]$ for points $\Spec k\to X$, one sees $f$ must be surjective on points. Let $A\subset X$ be a subset whose preimage $B\subset Y$ is closed. Then in particular $A=f(B)\subset X$ is pro-constructible. To show that $A$ is closed, it suffices to show that it is closed under specializations. This reduces us to the case that $X=\Spec V$ is the spectrum of a valuation ring, and we may assume that the generic point lies in $A$. As $A$ is pro-constructible, it is itself spectral, and hence has a closed point $\xi$. Replacing $X$ by the closure of $\xi$, we can assume that the generic point of $X$ is a closed point of $A$. This actually means $A=\Spec K$ is just the generic point, so $B=Y_K\subset Y_V$ is closed, contradicting the assumption.
\end{proof}

\begin{example}[An arc-cover that is not a universal submersion]\label{ex:arcnotsubmersion} We give an example of an arc-cover that is not a universal submersion, showing that universal submersions are strictly between arc-covers and v-covers. Let $K=k((t_1))((t_2))\ldots((t_n))\ldots$, a Laurent series ring in infinitely many variables, with its natural $\mathbb Z^{\mathbb N}$-valued valuation (with the lexicographic ordering), and let $V\subset K$ be its valuation ring. Then $X:=\Spec V = \{s_0,s_1,\ldots,s_n,\ldots,\eta\}$ has a generic point $\eta$, and $s_n$ specializes to $s_m$ if and only if $n\geq m$. Each specialization from $s_{n+1}$ to $s_n$ is covered by the rank $1$ valuation ring $V_n=k((t_1))\ldots((t_{n-1}))[[t_n]]$, so letting $Y=\Spec(\prod_{n\geq 0} V_n)$, the map $Y\to X$ is an arc-cover. (Note that there are no rank $1$ specializations from $\eta$ to any $s_n$, and that $\eta$ lies in the image of $Y$, as the image is pro-constructible.) Note that there is a natural map from $Y$ to $\beta\mathbb N$, the Stone-\v{C}ech compactification of $\mathbb N$ -- this is always true for $\Spec(\prod_{n\geq 0} R_n)$ for rings $R_n$. Now $\mathbb N\subset \beta \mathbb N$ is open, and its preimage in $Y$ is $\bigsqcup_{n\geq 0} \Spec V_n$. This is actually also the preimage of $\Spec V\setminus \{\eta\}\subset \Spec V$: Indeed, under the composite map $Y\to X\to \{s_0,\ldots,s_m\}$ (collapsing all $s_n$, $n\geq m$, and $\eta$ to $s_m$), all of $\Spec(\prod_{n\geq m} V_n)$ maps to $s_m$, so the intersection of these subsets, which is exactly the preimage of $\beta\mathbb N\setminus \mathbb N$, maps to $\eta$.

Thus, we see that in this example the preimage of $\eta=\Spec K\subset X=\Spec V$ in $Y$ is closed, so $Y\to X$ is not a (universal) submersion.
\end{example}

\begin{theorem}\label{thm:etalesheavesdescend} Sending a qcqs scheme $X$ to the category of sheaves on $X_\et$ defines a stack of categories with respect to universal submersions, in particular a v-stack.

In particular, sending a qcqs scheme $X$ to the category of separated \'etale maps $Y\to X$ defines a stack of categories with respect to universal submersions, in particular a v-stack.
\end{theorem}

As the proof shows, the fully faithfulness part actually holds in the arc-topology.

\begin{proof} The second part is a consequence of the first: Indeed, then any descent datum for a separated \'etale map gives by descent some sheaf on $X_\et$, which is necessarily representable by an algebraic space \'etale over $X$, and by descent separated. By \cite[Tags 0417, 03XU]{StacksProject}, it is automatically representable a scheme over $X$. Thus, we can concentrate on the first part.

First, we prove fully faithfulness, so take two \'etale sheaves $\mathcal F$, $\mathcal G$ on $X$. This part will actually work in the arc-topology. We want to show that for an arc-cover $Y\to X$, any morphism $f: \mathcal F|_Y\to \mathcal G|_Y$ whose two pullbacks to $Y\times_X Y$ agree descends uniquely to $X$. As the category of \'etale sheaves is generated under colimits by representable sheaves, we can reduce to the case that $\mathcal F$ is representable by some \'etale $X$-scheme $X_i$. Replacing $X$ by $X_i$, we can assume that $\mathcal F=\ast$ is a point, in which case what we have to show is that any \'etale sheaf actually defines an arc-sheaf. Let $\mathcal G'$ be the \'etale sheaf on $X$, taking any \'etale $X'\to X$ to the sections of $\mathcal G(Y\times_X X')$ invariant under the descent datum, so we get a map $\mathcal G\to \mathcal G'$ of \'etale sheaves on $X$, that we need to prove is an isomorphism. It suffices to prove that it is an isomorphism after passing to stalks, so we can assume that $X$ is strictly henselian, and reduce to checking that it is an isomorphism on global sections. It is easy to see that the map $\mathcal G(X)\to \mathcal G(Y)$ is injective (for example, by pulling back to the closed point), so we need to see that any section $s\in \mathcal G(Y)$ invariant under the descent datum descends to $X$. Pulling back to the closed point of $X$, where we get an fpqc cover of a field, hence an ind-fppf cover (along which etale sheaves descend), we find a unique section $s_0\in \mathcal G(X)$ whose pullback to $Y$ agrees with $s$ after pullback to the closed point of $X$. We need to see that $s$ is the pullback of $s_0$. This can be checked on geometric points. Thus, it suffices to check this after pullback to geometric points of $X$; so connecting these to the special point of $X$, we can assume that $X=\Spec V$ is the spectrum of a valuation ring (with algebraically closed fraction field $K$). As $\mathcal G$ takes cofiltered limits of affine schemes to filtered colimits, we can assume that $Y\to X$ is of finite presentation (and an arc-cover), in which case $Y\to X$ splits after pullback to a finite chain of locally closed $\Spec V_i\subset \Spec V$, connecting the special point to the generic point. This finishes the argument.

It remains to prove effectivity of descent. Let $Y\to X$ be a universally submersive cover, and let $\mathcal G$ be any \'etale sheaf on $Y$ with a descent datum to $X$. Sending any $X$-scheme $X'$ to the sections of $\mathcal G(X'\times_X Y)$ invariant under the descent datum defines an arc-sheaf $\mathcal F$ on $X$ whose pullback to $Y_{\mathrm{arc}}$ is the pullback of $\mathcal G$ on $Y_\et$. We need to see $\mathcal F$ comes via pullback from its restriction to the \'etale site of $X$. Lemma~\ref{lem:recognizeetalesheaf} gives an equivalent criterium: Commutation with filtered colimits, invariance under change of separably closed base field, and invariance under passing from a strictly henselian ring to its closed point. The commutation with filtered colimits follows via descent from the same property of the pullback of $\mathcal G$ to $Y_{\mathrm{arc}}$. For the invariance under change of separably closed base field, we can assume that $X$ is a geometric point. In that case, we can assume that also $Y$ is a geometric point, in which case $\mathcal G$ is merely a set, and as $Y\times_X Y$ is connected, there are no nontrivial descent data, so the descent is trivial.

Now we check the injectivity in part (iii) of Lemma~\ref{lem:recognizeetalesheaf}, so we can now assume that $X$ is strictly henselian with closed point $x$. Take any $s,t\in \mathcal F(X)$. The locus where $s=t$ defines an arc-subsingleton sheaf, and after pullback to $Y$ it is representable by an open subset of $Y$. As $Y\to X$ is a submersion (and an arc-cover), this implies that it is representable by an open subset of $X$. If $s=t$ over $x$, then this open subset must be all of $X$, hence $s=t$, giving the injectivity.

It remains to prove surjectivity, and for this we may assume that $X=\Spec V$ is the spectrum of an absolutely integrally closed valuation ring. Pick any $s\in \mathcal F(x)$ and assume that $s$ does not lift to $\mathcal F(X)$. By Zorn's lemma (and the commutation with filtered colimits), we can assume that $s$ does lift to all proper closed subsets $Z\subset X$. But we know that $Y_K\subset Y_V$ is not a closed immersion, so we can find an absolutely integrally closed valuation ring $W$ with a map $\Spec W\to Y$ whose generic point maps to the generic point of $X=\Spec V$, but whose special point does not map to the generic point of $X$. By arc-descent on $X$, we may replace $X$ by the image of $\Spec W\to \Spec V$. So we can assume that $Y=\Spec W$, where $V\to W$ is a faithfully flat extension of absolutely integrally closed valuation rings. In particular, $s$ extends uniquely to a section over $\Spec W$, and its two pullbacks to $\Spec W\times_{\Spec V}\Spec W$ agree as in fact $\mathcal F(\Spec W\times_{\Spec V}\Spec W)\cong \mathcal F(\Spec W)$ by Lemma~\ref{lem:invarianceaicvaluationring} below (and thus in turn agrees with the sections over the closed point).
\end{proof}

\begin{lemma}\label{lem:recognizeetalesheaf} Let $X$ be a qcqs scheme, and let $\mathcal F$ be an arc-sheaf over $X$. Then $\mathcal F$ comes via pullback from a sheaf on $X_\et$ if and only if the following conditions are satisfied.
\begin{enumerate}
\item[{\rm (i)}] The arc-sheaf $\mathcal F$ is finitary, i.e.~for any cofiltered system $X_i=\Spec A_i$ of affine $X$-schemes with limit $X=\Spec A$, the map
\[
\varinjlim_i \mathcal F(X_i)\to \mathcal F(X)
\]
is a bijection.
\item[{\rm (ii)}] For any map $\Spec K'\to \Spec K$ of geometric points over $X$, the map
\[
\mathcal F(\Spec K)\to \mathcal F(\Spec K')
\]
is a bijection.
\item[{\rm (iii)}] For any strictly henselian $X$-scheme $Z$ with closed point $z$, the map $\mathcal F(Z)\to \mathcal F(z)$ is a bijection.
\end{enumerate}

Moreover, it suffices to verify (iii) in the restricted case where $Z$ is the spectrum of an absolutely integrally closed valuation ring.
\end{lemma}

We note that arc-sheaves are automatically invariant under universal homeomorphisms, in particular the difference between separably closed fields and algebraically closed fields is not relevant here.

\begin{proof} Clearly the conditions are necessary. Conversely, let $\mathcal F'$ be the pushforward of $\mathcal F$ to the \'etale site; we have to see that for all $g: Y\to X$, the map $g^\ast \mathcal F'\to \mathcal F|_{Y_\et}$ is an isomorphism. This can be checked on stalks, so using (i) we can assume that $Y$ is strictly henselian. Let $\overline{y}$ be the closed point of $Y$, mapping to a geometric point $\overline{x}$ of $X$. Then
\[(g^\ast \mathcal F')_{\overline{y}} = \mathcal F'_{\overline{x}}= \mathcal F(X_{\overline{x}})\cong \mathcal F(\overline{x})\cong \mathcal F(\overline{y})\cong \mathcal F(Y),
\]
using (iii) for $X_{\overline{x}}$, (ii), and (iii) for $Y$, respectively. This gives the first part.

Next, assume we know only (i), (ii), the injectivity in (iii), and surjectivity in (iii) when restricted to absolutely integrally closed valuation rings. Take any strictly henselian $X$-scheme $(Z,z)$; we want to show that $\mathcal F(Z)\to \mathcal F(z)$ is surjective. Fix some section $s\in \mathcal F(z)$ and assume that $s$ does not lift to $\mathcal F(Z)$. We can replace $X$ by $Z$ and assume that $X$ is strictly henselian. Consider the partially ordered set of all closed subschemes $Z\subset X$ (necessarily strictly henselian) such that $s$ does not lift to $\mathcal F(Z)$. Using (i), we see that we can apply Zorn's lemma and find a minimal $Z$. Then $Z$ is irreducible, as otherwise $Z=Z_1\cup Z_2$ is a union of two proper closed subschemes to which $s$ lifts, in which case $s$ lifts to $Z$ as $Z_1\sqcup Z_2\to Z$ is an arc-cover (and we have agreement of the lifts $s_1$ (of $s$ to $Z_1$) and $s_2$ (of $s$ to $Z_2$) over $Z_1\cap Z_2$, by the injectivity already proved). Replacing $X$ by $Z$ again, we can assume that $X$ is the spectrum of a strictly henselian domain.

Similarly, if $X'\to X$ is finite, then necessarily $X'$ is a finite disjoint union of strictly henselian schemes whose closed points lie over $\overline{x}$, and in particular we get an injection $\mathcal F(X')\hookrightarrow \mathcal F(X'\times_X \overline{x})$. Applying this observation to $X'\times_X X'$ in case $X'\to X$ is surjective, we see that it suffices to see that $s|_{X'\times_X \overline{x}}$ extends to $X'$. Passing to a limit again, we can assume that $X$ is the spectrum of an absolutely integrally closed local domain.

Now let $g: X'\to X$ be a blowup of $X$. Assume that for all geometric points $\overline{x'}$ of $X'\times_X \overline{x}$, the section $s|_{\overline{x'}}$ extends to $X'_{\overline{x'}}$. Then $s$ extends to a global section of the pullback of $\mathcal F|_{X'_\et}$ to $(X'\times_X \overline{x})_\et$. By proper base change \cite[Tag 0A0C]{StacksProject}, this gives a unique section $s$ of $\mathcal F(X')$. Applying a similar argument to $X'\times_X X'$ and using that $X'\to X$ is an arc-cover then shows that the section of $\mathcal F(X')$ descends to $\mathcal F(X)$.

Note that the locus of geometric points $\overline{x'}$ of $X'\times_X \overline{x}$ where $s|_{\overline{x'}}$ extends to $X'_{\overline{x'}}$ defines an open subspace of $X'\times_X \overline{x}$ (using again condition (i)), so for each blowup we get a nonempty closed subset of $X'\times_X \overline{x}$ where the section $s$ does not lift. By Tychonoff, the inverse limit of these closed subsets, taken over all blowups $X'$ of $X$, is nonempty still. Picking a point in the intersection will then define a local ring which is an absolutely integrally closed valuation ring, where $s$ still does not extend. This contradicts our assumption that (iii) holds for spectra of absolutely integrally closed valuation rings, giving surjectivity in (iii) in general.

Finally, assume we know only (i), (ii), and the bijectivity in (iii) when restricted to absolutely integrally closed valuation rings. By what we have already proved, we need to see that this gives injectivity in (iii) in general. Suppose given $Z$ strictly henselian and two sections $s,t \in \mathcal{F}(Z)$. The locus where $s = t$ defines an arc subsingleton sheaf. As it is a subsingleton sheaf, the injectivity in (iii) is automatically satisfied, so by what we have proved so far, this locus defines an \'etale subsingleton sheaf over the strictly henselian scheme $Z$. Thus, if the locus contains the closed point, it must be everything, giving the injectivity of $\mathcal F(Z)\to \mathcal F(z)$.
\end{proof}

The proof of the following lemma makes a somewhat strange reduction from sheaves of sets to sheaves of abelian groups killed by some integer invertible on the scheme.

\begin{lemma}\label{lem:invarianceaicvaluationring} Let $V\to W$ be a faithfully flat extension of absolutely integrally closed valuation rings. Let $X$ be a scheme over $\Spec V$ with base change $f: X\times_{\Spec V} \Spec W\to X$. Let $\mathcal F$ be an \'etale sheaf (of sets) on $X$. Then the map
\[
\mathcal F(X)\to (f^\ast \mathcal F)(X\times_{\Spec V} \Spec W)
\]
is a bijection.
\end{lemma}

\begin{proof} First note that the map is injective, as $X\times_{\Spec V}\Spec W\to Y$ is faithfully flat (in particular, an arc-cover), so one has to prove surjectivity. Equivalently, any section of $(f^\ast \mathcal F)(X\times_{\Spec V}\Spec W)$ is invariant under the descent datum. This can be checked after embedding $\mathcal F$ into the free sheaf of $\mathbb F_\ell$-modules $\mathbb F_\ell[\mathcal F]$ on $\mathcal F$, for some chosen prime $\ell$ invertible in $V$. Thus, we can assume that $\mathcal F$ is an abelian torsion sheaf, killed by some prime $\ell$ invertible in $V$. Now the result follows from a theorem of Huber \cite[Corollary 4.2.7]{HuberBook}, which we have reproved in the previous section as Corollary~\ref{cor:invarianceofcohomology}.
\end{proof}

Combining this with the arc-descent results of Bhatt--Mathew \cite{BhattMathew}, we obtain the following result. Here we denote by $\mathcal D^+_{\mathrm{tor}}(S_\et)$ the bounded to the left derived $\infty$-category of torsion abelian sheaves on $S_\et$.

\begin{theorem}\label{thm:univsubmersivedescent} The association taking any qcqs scheme $S$ to $\mathcal D^+_{\mathrm{tor}}(S_\et)$ defines a sheaf of $\infty$-categories for the topology of universal submersions.
\end{theorem}

One can also formally deduce an unbounded variant by passing to left-completions. In particular, $S\mapsto \mathcal D_\et(S,\Lambda)$ in setting (A) defines a sheaf of $\infty$-categories for the topology of universal submersions.

\begin{proof} To prove fully faithfulness, we need to see that for any $A\in D^+_{\mathrm{tor}}(S_\et)$, the functor $T/S\mapsto R\Gamma(T,A|_T)$ defines a sheaf for the topology of universal submersions. In fact, it defines an arc-sheaf, by \cite[Theorem 5.4]{BhattMathew}. For effectivity of descent data, one can then reduce to the case that $A$ is concentrated in degree $0$. By Theorem~\ref{thm:etalesheavesdescend}, it descends as a sheaf of sets, but the group structure descends as well, and is necessarily torsion. 
\end{proof}

\section{Relative perversity}

Finally, we can prove our results on relative perversity. Recall the statement of our main theorem:

\begin{theorem}\label{thm:maintext} Let $f: X\to S$ be a finitely presented map of qcqs $\mathbb Z[\tfrac 1\ell]$-schemes. Consider any of the settings (A), (B) and (C). In case (B), assume moreover that $\Lambda$ is regular (in the weak sense that any truncation of a perfect complex is still perfect). In case (C), assume that any constructible subset of $S$ has finitely many irreducible components.

There is a $t$-structure $({}^{p/S}D^{\leq 0},{}^{p/S} D^{\geq 0})$ on $D(X)$, called the relative perverse $t$-structure, with the following properties.
\begin{enumerate}
\item[{\rm (i)}] An object $A\in D(X)$ lies in ${}^{p/S}D^{\leq 0}$ (resp.~${}^{p/S}D^{\geq 0}$) if and only if for all geometric points $\overline{s}\to S$ with fibre $X_{\overline{s}} = X\times_S \overline{s}$, the restriction $A|_{X_{\overline{s}}}\in D(X_{\overline{s}})$ lies in ${}^p D^{\leq 0}$ (resp.~${}^p D^{\geq 0}$), for the usual (absolute) perverse $t$-structure.
\item[{\rm (ii)}] For any map $S'\to S$ of schemes (with $S'$ satisfying the same condition as $S$, in case (C)) with pullback $X'=X\times_S S'\to X$, the pullback functor $D(X)\to D(X')$ is $t$-exact with respect to the relative perverse $t$-structures.
\item[{\rm (iii)}] In case (A), the full sub-$\infty$-categories ${}^{p/S} \mathcal D^{\leq 0},{}^{p/S} \mathcal D^{\geq 0}\subset \mathcal D(X)$ are stable under all filtered colimits.
\end{enumerate}
\end{theorem}

\begin{remark}\label{rem:hypotheses} The hypothesis in case (B) is clearly necessary, already when $X=S=\Spec K$ is a geometric point. The hypothesis in case (C) is not quite optimal, but we note that it is definitely necessary to assume that all constructible subsets of $S$ have only finitely many \emph{connected} components. Indeed, take $X=S$, in which case we want a naive $t$-structure on $D_\cons(S,\mathcal O_L)$ or $D_\cons(S,L)$. Assume that $S$ has some constructible subset with infinitely many connected components. Replacing $S$ by this constructible subset, we can assume that there is a surjective continuous map
\[
S\to \{0,1,2,\ldots,\infty\}.
\]
Now one can look at the dualizable complex
\[
[\mathbb Z_\ell\to \mathbb Z_\ell]
\]
where the map multiplies by $\ell^n$ in the fibre of $S$ over $n$ (where $\ell^\infty=0$). It is easy to see that the truncations of this complex are not constructible (in particular, the kernel of this complex is trivial except in the fibre of $S$ over $\infty$).

One can show that when $S$ is purely of characteristic $0$, this weaker condition is in fact sufficient. To show this, one argues as in the proof below, but using the version of Theorem~\ref{thm:perverseULAmaintext}~(i) given in Remark~\ref{rem:perverseULAchar0}.

On the other hand, when $S$ contains points of positive characteristic, this strengthening of Theorem~\ref{thm:perverseULAmaintext}~(i) fails, and in fact one can find nonzero universally locally acyclic perverse sheaves that vanish in some closed fibers, using Artin-Schreier covers. (We thank Haoyu Hu and Enlin Yang for showing us an explicit example of such a sheaf.) Combining such constructions with the above counterexample can be used to show that having only finitely many irreducible components is essentially necessary.
\end{remark}

We will freely use in the proof that such $t$-structures exist in case $S=\Spec K$ is the spectrum of an algebraically closed field $K$. We advise the reader to read only the proofs in settings (A) and (B) on first reading; in fact, this is necessary to avoid vicious circles.

\begin{proof} Parts (ii) and (iii) are formal consequences of (i). For part (i), assume first that we are in setting (A) or (B). In setting (A), we can formally define a $t$-structure on $\mathcal D_\et(X,\Lambda)$ by taking the connective part ${}^{p/S} \mathcal D^{\leq 0}$ to consist of all $A\in \mathcal D_\et(X,\Lambda)$ such that for all geometric points $\overline{s}\to S$ the restriction $A|_{X_{\overline{s}}}\in {}^p \mathcal D^{\leq 0}(X_{\overline{s}},\Lambda)$, by applying \cite[Proposition 1.4.4.11]{LurieHA}. We have to show that the corresponding coconnective part has the stated characterization, and that it induces a $t$-structure on the constructible objects in case $\Lambda$ is regular.

We start by analyzing the case where $S=\Spec V$ is the spectrum of an absolutely integrally closed valuation ring $V$ of rank $1$. Let $j: X_\eta\subset X$ and $i: X_s\subset X$ be the open and closed immersion of generic and special fibre. Then it follows formally from the definition of the $t$-structure that $A\in {}^{p/S} \mathcal D^{\geq 0}$ if and only if $A|_{X_\eta}\in {}^p \mathcal D^{\geq 0}(X_\eta,\Lambda)$ and $Ri^! A\in {}^p \mathcal D^{\geq 0}(X_s,\Lambda)$. We have to see that these conditions are equivalent to the two conditions $A|_{X_\eta}\in {}^p \mathcal D^{\geq 0}_\et(X_\eta,\Lambda)$ and $i^\ast A\in {}^p \mathcal D^{\geq 0}(X_s,\Lambda)$. Thus, assume $A|_{X_\eta}\in {}^p \mathcal D^{\geq 0}_\et(X_\eta,\Lambda)$. Then we have a triangle
\[
Ri^! A\to i^\ast A\to i^\ast Rj_\ast(A|_{X_\eta})
\]
in $\mathcal D_\et(X_s,\Lambda)$. Thus, it suffices to show that $i^\ast Rj_\ast(A|_{X_\eta})\in {}^p \mathcal D^{\geq 0}_\et(X_s,\Lambda)$. This follows from the perverse $t$-exactness of nearby cycles, Lemma~\ref{lem:nearbycyclestexact} below.

In the case $S=\Spec V$ is the spectrum of an absolutely integrally closed valuation ring $V$ of rank $1$, it remains to show that relative perverse truncation preserves constructible objects in case $\Lambda$ is regular. But constructibility can be checked fibrewise on $S$, and relative perverse truncation commutes with passing to fibres by what we have already established. Thus, the claim reduces to the geometric fibres, where it is standard.

Next, we show that there is a perverse $t$-structure on the full $\infty$-subcategory $\mathcal D_{\cons,\tor}(X,\mathbb Z_\ell)\subset \mathcal D_\cons(X,\mathbb Z_\ell)$ of torsion constructible $\mathbb Z_\ell$-complexes. We observe that as the desired $t$-structure automatically behaves well with respect to base change in $S$, it suffices to construct it locally on $S$ as long as the $\infty$-categories satisfy descent in $S$. By Theorem~\ref{thm:arcdescentD}, this is the case for arc-covers. In particular, using v-descent we can reduce to the case that all connected components of $S$ are spectra of absolutely integrally closed valuation rings.

Assume first that $S$ is connected, so the spectrum of an absolutely integrally closed valuation ring $V$. In that case, by approximation, we can reduce to the case $V$ is of finite rank, and then by arc-descent to the case that $V$ is of rank $1$. We have already handled this case, noting that in this case the $\mathrm{Ind}$-category of $\mathcal D_{\cons,\tor}(X,\mathbb Z_\ell)$ can be identified with the torsion objects in $\mathcal D(X_\et,\mathbb Z_\ell)$, to which the arguments above apply. In general, observe first that if $A\in \mathcal D_{\cons,\tor}(X,\mathbb Z_\ell)$ and $B\in \mathcal D_{\cons,\tor}(X,\mathbb Z_\ell)$ such that all geometric fibres of $A$ are in ${}^p \mathcal D^{\leq 0}$ and all geometric fibres of $B$ are in ${}^p \mathcal D^{\geq 1}$, then $\mathrm{Hom}(A,B)=0$. Indeed, take any map $f: A\to B$. To see that $f=0$, it suffices to show that $f$ vanishes after pullback to all connected components of $S$. But here it follows from the results on the $t$-structure. Thus, to show that these subcategories define a $t$-structure, it suffices to construct the truncations of any $A\in \mathcal D_{\cons,\tor}(X,\mathbb Z_\ell)$. Fix some $c\in \pi_0 S$, giving rise to a connected component $S_c\subset S$. Using the relative perverse $t$-structure on $X_c=X\times_S S_c\to S_c$, we can find a triangle
\[
{}^{p/S_c}\tau^{\leq 0} A_c\to A_c\to {}^{p/S_c}\tau^{\geq 1} A_c
\]
where $A_c=A|_{X_c}$. As everything is constructible, this triangle extends to a similar triangle over an open and closed neighborhood $S'\subset S$ of $S_c$. By Lemma~\ref{lem:perverseamplitude} below, the resulting triangle still reduces to the relative perverse truncation in all fibres, after possibly shrinking $S'$. Thus, the desired truncation functors can be defined on $A$ at least locally on $S$, but then by uniqueness also globally. This finishes the proof of the existence of the $t$-structure on $\mathcal D_{\cons,\tor}(X,\mathbb Z_\ell)$.

In particular, passing to $\mathrm{Ind}$-categories when all connected components of $S$ are absolutely integrally closed valuation rings, we get a $t$-structure on the full $\infty$-subcategory of torsion objects in $\mathcal D_\et(X,\mathbb Z_\ell)$, and then by passing to $\Lambda$-modules and v-descent we get the $t$-structure in setting (A). In setting (B), it remains to prove that the perverse truncations preserve $\mathcal D_\cons(X,\Lambda)$. For this, we can again assume that all connected components of $S$ are absolutely integrally closed valuation rings. Using Lemma~\ref{lem:perverseamplitude}, we can reduce to the connected components. By approximation, we can then also assume that these are of finite rank. In that case, constructibility can be checked on geometric fibres; thus, the claim reduces to the case where $S$ is a geometric point, where the result is standard.

In setting (C), we can formally reduce to the case that $L$ is a finite extension of $\mathbb Q_\ell$, and the case of rational coefficients follows formally from the case of integral coefficients by inverting $\ell$. Now we first show that if $A,B\in \mathcal D_\cons(X,\mathcal O_L)$ have the property that all geometric fibres $A|_{X_{\overline{s}}}\in {}^p \mathcal D^{\leq 0}_\cons(X_{\overline{s}},\mathcal O_L)$ (resp.~$B|_{X_{\overline{s}}}\in {}^p \mathcal D^{\geq 1}_\cons(X_{\overline{s}},\mathcal O_L)$), then $\Hom(A,B)=0$. To see this, write $B$ as the derived limit of the reductions $B_n=B/^{\mathbb L} \ell^n$. Then $B_n$ lies in the corresponding category of type (B), and lies in ${}^{p/S} \mathcal D^{\geq 0}_{\cons,\tor}(X,\mathcal O_L)$. We claim that the system ${}^{p/S}\mathcal H^0(B_n)$ of relatively perverse sheaves on $X/S$ is pro-zero. More precisely, fix a constructible stratification of $S$ over which $B$ becomes universally locally acyclic, and let $\overline{s}_1,\ldots,\overline{s}_r$ be the geometric generic points of the strata of $S$ (of which there are only finitely many by assumption). Choose some $N$ such that $\ell^N$ kills the torsion part of ${}^p \mathcal H^1(B|_{X_{\overline{s}_i}})\in \mathrm{Perv}(X_{\overline{s}_i})$ for all $i=1,\ldots,r$. Then we claim that the transition map ${}^{p/S}\mathcal H^0(B_{N+n})\to {}^{p/S}\mathcal H^0(B_n)$ is zero for all $n$. This can be checked over the stratification, and then over the closure of each irreducible component, and then by Theorem~\ref{thm:perverseULAmaintext}~(i), it can be checked in the geometric fibres $X_{\overline{s}_i}$ for $i=1,\ldots,r$, where it follows from our choice of $N$.

Thus, we see that
\[
\Hom(A,B) = \varprojlim_n \Hom(A,B_n) = \varprojlim_n \Hom(A,{}^{p/S}\mathcal H^0(B_n))=0,
\]
as desired. It remains to see that any $A\in \mathcal D_\cons(X,\mathcal O_L)$ admits a triangle
\[
{}^{p/S} \tau^{\leq 0}A\to A\to {}^{p/S} \tau^{\geq 1} A
\]
where the first term is fibrewise in ${}^p \mathcal D^{\leq 0}$, and the last term is fibrewise in ${}^p \mathcal D^{\geq 1}$. This can be obtained from the similar triangle for $A_n=A/^{\mathbb L} \ell^n$ by passing to an inverse limit, using a similar argument as above for controlling $\ell$-power torsion.
\end{proof}

The following lemmas were used in the proof.

\begin{lemma}\label{lem:nearbycyclestexact} Let $S=\Spec V$ be the spectrum of an absolutely integrally closed valuation ring $V$ of rank $1$, and let $X$ be a finite type $S$-scheme. Let $j: X_\eta\subset X$ and $i: X_s\subset X$ be the open and closed immersion of generic and special fibre. Then for any torsion $\mathbb Z_\ell$-algebra $\Lambda$, the nearby cycles functor
\[
R\psi = i^\ast Rj_\ast: \mathcal D_\et(X_\eta,\Lambda)\to \mathcal D_\et(X_s,\Lambda)
\]
is $t$-exact with respect to the absolute perverse $t$-structures on source and target.
\end{lemma}

This is the key fact about the usual perverse $t$-structure that we use.

\begin{proof} Forgetting the $\Lambda$-module structure, we can reduce to $\mathcal D_{\cons,\tor}(-,\mathbb Z_\ell)$. As $R\psi$ commutes with Verdier duality and Verdier duality exchanges ${}^p \mathcal D^{\leq 0}_{\cons,\tor}(-,\mathbb Z_\ell)$ and ${}^p \mathcal D^{\geq 1}_{\cons,\tor}(-,\mathbb Z_\ell)$, it suffices to show that $R\psi$ takes ${}^p \mathcal D^{\leq 0}_{\cons,\tor}(X_\eta,\mathbb Z_\ell)$ into ${}^p \mathcal D^{\leq 0}_{\cons,\tor}(X_s,\mathbb Z_\ell)$. But this follows from Artin vanishing and \cite[R\'eciproque 4.1.6]{BBDG}.
\end{proof}

\begin{lemma}\label{lem:perverseamplitude} Let $f: X\to S$ be a finitely presented map of qcqs $\mathbb Z[\tfrac 1\ell]$-schemes, and let $A\in \mathcal D_{\cons,\tor}(X,\mathbb Z_\ell)$ or $A\in\mathcal D_\cons(X,\Lambda)$ in setting (B) with $\Lambda$ regular. The subset $S^{\leq 0}\subset S$ (resp.~$S^{\geq 0}\subset S$) of all points $s\in S$ for which $A|_{X_s}\in {}^p \mathcal D^{\leq 0}$ (resp.~$A|_{X_s}\in {}^p \mathcal D^{\geq 0}$) is a constructible subset of $S$.
\end{lemma}

\begin{proof} The case of $S^{\leq 0}$ is easy: By passing to a stratification of $X$, this case easily reduces to the case that $A$ is locally constant and $X$ is smooth and equidimensional over $S$, where it is clear.

Using Theorem~\ref{thm:nearbycycles} in the case of fields (where it says that all constructible complexes are universally locally acyclic) and Proposition~\ref{prop:ULAfinitaryarc} in order to spread information at points to constructible subsets, we see that there is a constructible stratification of $S$ over which $A$ becomes universally locally acyclic. Passing to this stratification, we can assume that $A$ is universally locally acyclic. If $\Lambda=\mathbb Z_\ell$, we can now use that passing to relative Verdier duals commutes with any pullback, and exchanges ${}^p \mathcal D^{\leq 0}_{\cons,\tor}(X_{\overline{s}},\mathbb Z_\ell)$ and ${}^p \mathcal D^{\geq 1}_{\cons,\tor}(X_{\overline{s}},\mathbb Z_\ell)$. In the case $\Lambda$-coefficients, let $I$ be an injective $\Lambda$-module such that $\Hom_\Lambda(-,I)$ is conservative. Then also the formation of $R\sHom(A,Rf^! I)$ commutes with any pullback, and moreover it is given by $\mathbb D_{X/S}(A)\dotimes_\Lambda I$, which becomes locally constant over a constructible stratification (although not with perfect fibres, but this does not matter for the argument). Moreover, in each fibre the functor $A\mapsto R\sHom(A,Rf^! I)$ from $\mathcal D_\cons(X_{\overline{s}},\Lambda)^{\mathrm{op}}$ to $\mathcal D_\et(X_{\overline{s}},\Lambda)$ is faithful and $t$-exact for the perverse $t$-structure; this gives the result in general.
\end{proof}

Using the relative perverse $t$-structure, we have the following relative version of Artin vanishing. We note the strong hypothesis on the base scheme. The essential content of this proposition is due to Gabber \cite{GabberOberwolfach2020}.

\begin{proposition}\label{prop:artinvanishing} Let $S=\Spec V$ be the spectrum of an absolutely integrally closed valuation ring $V$, and let $f: Y\to X$ be an affine map of schemes of finite presentation over $V$. Then
\[
Rf_\ast: D(Y)\subset D(Y_\proet,\Lambda)\to D(X_\proet,\Lambda)
\]
takes values in $D(X)\subset D(X_\proet,\Lambda)$ and is right t-exact for the relative perverse $t$-structure, in any of the settings considered in Theorem~\ref{thm:maintext}. Moreover, if $S'=\Spec V'$ is of the similar form and $g: S'\to S$ is flat, with pullback $f': Y'\to X'$ (with $g_Y: Y'\to Y$ and $g_X: X'\to X$), then the base change map
\[
g_X^\ast Rf_\ast\to Rf'_\ast g_Y^\ast
\]
of functors $D(Y)\to D(X')$ is an isomorphism.
\end{proposition}

We note that over any base $S$, and for any affine map $g: Y\to X$ of finitely presented $S$-schemes, the functor $Rg_!$ is left t-exact for the relative perverse $t$-structure; this assertion immediately reduces to the statement over geometric points. By contrast, Proposition \ref{prop:artinvanishing} does not formally reduce to its version over geometric points, and does not hold over more general bases. (We warn the reader that over $S$ as in the proposition, Verdier duality is not a perfect duality; in fact, it vanishes on all sheaves whose restriction to the generic fibre vanishes. Thus, one cannot control $Rg_\ast$ in terms of $Rg_!$.)

\begin{proof} Setting (C) with rational coefficients reduces to setting (C) with integral coefficients by inverting $\ell$, and this in turn reduces to setting (B). Forgetting the $\Lambda$-module structure, all statements except for preservation of constructibility reduce to the case of $\mathcal D_{\cons,\tor}(-,\mathbb Z_\ell)$.

Let us first handle the base change result. By checking sections over all \'etale $X'$-schemes, it suffices to show that the map
\[
R\Gamma(X',g_X^\ast Rf_\ast A)\to R\Gamma(Y',g_Y^\ast A)
\]
is an isomorphism. But Corollary~\ref{cor:invarianceofcohomology} reduces this to $R\Gamma(X,Rf_\ast A) = R\Gamma(Y,A)$ which is clear.

Now we have to show that $Rf_\ast$ preserves constructibility and is right t-exact. By approximation, we can assume that $V$ is of finite rank, and that the sheaf is concentrated in one fibre over $S$. As in the proof of Lemma~\ref{lem:finitecohomdim}, we can then use arc-excision to reduce to the case that $V$ is of rank $1$ (and a sheaf concentrated on the generic fibre).

To show preservation of constructibility, we can now make a d\'evissage to sheaves concentrated on the special fibre, and sheaves $\ast$-extended from the generic fibre. The first case reduces to the known assertion when $S$ is a geometric point, and the second case also reduces to this assertion on the generic fibre, together with Theorem~\ref{thm:nearbycycles}.

It remains to prove right t-exactness, in the case that $V$ is of rank $1$ and the sheaf is concentrated in the generic fibre. We first handle the case that $X=S$ and $Y$ is an affine curve over $S$. In that case, we have to prove that the cohomological dimension of $Y$ is $1$. We can assume that $\mathcal F=j_! L$ for some open immersion $j: V\subset Y$ contained in the generic fibre and some local system $L$ on $V$; we can also assume that $V$ is smooth. Let $W\to V$ be a finite \'etale $G$-torsor trivializing $L$ and let $j': W\subset Z$ be the normalization of $Y$ in $W$. Then $R\Gamma(Y,\mathcal F)$ can be identified with the $G$-homology on $R\Gamma(Z,j'_! L|_W)$. Thus, we can assume that $L$ is trivial, and then reduce to $L=\mathbb F_\ell$. Moreover, we can assume that the generic fibre of $Z$ is smooth. Let $j_Z: Z_\eta\to Z$ be the open immersion, and $i_Z: Z_s \to Z$ the closed immersion of the special fiber. Then the cone of $j'_! \mathbb F_\ell\to j_{Z!} \mathbb F_\ell$ is a skyscraper sheaf at the finitely many points of $Z_\eta\setminus W$, all of which are geometric points, and so we reduce to the sheaf $j_{Z!}\mathbb F_\ell$. This sheaf sits in a triangle \[ j_{Z!}\mathbb F_\ell \to Rj_{Z \ast}\mathbb F_\ell \to i_{Z \ast} i_{Z}^{\ast} Rj_{Z \ast}\mathbb F_\ell \to, \]so applying $R\Gamma(Z,-)$ gives a triangle
\[ R\Gamma(Z, j_{Z!}\mathbb F_\ell) \to R\Gamma(Z_\eta, \mathbb F_\ell) \to R\Gamma(Z_s, i_{Z}^{\ast} Rj_{Z \ast} \mathbb F_\ell) \to. \] Using Lemma \ref{lem:nearbycyclestexact} together with Artin vanishing in the generic and special fibers, we see that the two rightmost terms of this triangle are concentrated in degrees $\leq 1$. This reduces us to the surjectivity of the map $H^1(Z_\eta, \mathbb F_\ell) \to H^1(Z_s, i_{Z}^{\ast} Rj_{Z \ast} \mathbb F_\ell)$, which is Lemma \ref{lem:runge} below.

The rest of the following argument is similar to the proof of Artin vanishing, and inspired by \cite[Th\'eor\`eme 2.4]{IllusieVariation}. We argue by induction on $d(A)$, where for $A\in {}^{p/S} D_{\cons,\tor}^{\leq 0}(Y,\mathbb Z_\ell)$, we denote by $d(A)$ the relative dimension of the closure of the support of $A$. Here, the relative dimension of a scheme of finite type over $S$ is the maximum of the dimension of its two fibres. Choosing a closed immersion, we can assume that $Y=\mathbb A^n_X$, and then by induction we reduce to $Y=\mathbb A^1_X$. Let $j_Y: Y_\eta\subset Y$ and $i_Y: Y_s\subset Y$ be the inclusion of the generic and special fibre (and we will use similar notation for $X$). Using the triangle
\[
j_{Y!} A|_{Y_\eta}\to A\to i_{Y\ast} A|_{Y_s}
\]
and Artin vanishing in the special fibre, we reduce to $A=j_{Y!} A_0$ for some $A_0\in {}^p D_{\cons,\tor}^{\leq 0}(Y_\eta,\mathbb Z_\ell)$.

We can replace $X$ by a strict henselization at one of its points, which we can assume to lie in the special fibre (as the result is known in the generic fibre). In fact, we can assume that it is a closed point of the special fibre. Indeed, if not, we can find a map $X\to \mathbb A^1_S$ sending $x$ to the generic point of the special fibre, which on strict henselizations will factor over the strict henselization of $\mathbb A^1_S$ at the generic point of the special fibre, which is the spectrum of a valuation ring $W$ whose fraction field has absolute Galois group pro-$p$, where $p$ is the residue characteristic of $V$. As pro-$p$-extensions are insensitive to the desired vanishing, we can then replace $V$ by $W$ and argue by induction. Let $x\in X$ denote the closed point of $X$. We have to show that
\[
R\Gamma(\mathbb A^1_X,A)\in D^{\leq 0}(\mathbb Z_\ell).
\]

Now consider the cartesian diagram
\[\xymatrix{
\mathbb A^1_X\ar[r]^j\ar[d]^{g^\circ} & \mathbb P^1_X\ar[d]^g\\
\mathbb A^1_S\ar[r]^{j'} & \mathbb P^1_S.
}\]
Then by proper base change
\[
R\Gamma(\mathbb A^1_X,A) = R\Gamma(\mathbb P^1_X,Rj_\ast A)=R\Gamma(\mathbb P^1_x,(Rj_\ast A)_{\mathbb P^1_x}).
\]
Moreover, $(Rj_\ast A)|_{\mathbb P^1_x}$ is concentrated on $x\times\{\infty\}$, as $A=j_{Y!} A_0$. It follows that
\[
R\Gamma(\mathbb A^1_X,A) = (Rj_\ast A)_{x\times\{\infty\}}.
\]
Taking strict henselizations at $x \times \{\infty \} \in \mathbb P^1_X$ and $s \times \{ \infty \} \in \mathbb P^1_S$ on the right-hand side of the previous cartesian diagram, we get a cartesian diagram
\[\xymatrix{
U\ar[r]^{u}\ar[d]^{h^\circ} & Z\ar[d]^{h}\\
V\ar[r]^{v} & T
}\]
and
\[
(Rj_\ast A)_{x\times\{\infty\}} = R\Gamma(U,A) = R\Gamma(V,Rh^\circ_\ast A).
\]
Now $h^\circ: U\to V$ is a map of affine schemes essentially of finite type over $S$, and $V$ does not map to any closed points of $\mathbb A^1_s$. It follows from the inductive hypothesis (and passage to limits) that $Rh^\circ_\ast A\in {}^{p/S} D^{\leq 0}_\tor(V,\mathbb Z_\ell)$, where we interpret the latter statement in the loose sense that all the stalks sit in the expected degrees. Thus, it remains to show that for all $B\in {}^{p/S} D^{\leq 0}_\tor(V,\mathbb Z_\ell)$, one has
\[
R\Gamma(V,B)\in D^{\leq 0}(\mathbb Z_\ell).
\]
Now $V$ is a limit of affine curves over $S$, so by passage to limits, this reduces to the case of curves already handled.
\end{proof}

\begin{lemma}\label{lem:runge} Let $S=\Spec V$ be the spectrum of an absolutely integrally closed valuation ring of rank one. Let $X$ be an affine curve over $S$ with smooth generic fiber, with $j:X_\eta \subset X$ and $i: X_s \subset X$ the habitual inclusions. Then the natural map $H^1(X_\eta, \mathbb F_\ell) \to H^1(X_s, i^{\ast} Rj_{ \ast} \mathbb F_\ell)$ is surjective.
\end{lemma} 

\begin{proof} Let $\hat{X}$ be the formal completion of $X$ along its special fiber, and let $\hat{X}_\eta$ be the associated rigid generic fiber, so $\hat{X}_{\eta}$ is naturally an open affinoid subset of the rigid analytic curve $X_{\eta}^{\mathrm{an}}$. By \cite[Corollary 3.5.14]{HuberBook}, there is a natural isomorphism $H^1(X_s, i^{\ast} Rj_{ \ast} \mathbb F_\ell) \cong H^1(\hat{X}_{\eta}, \mathbb F_\ell)$, under which the map in the lemma identifies with the natural map $H^1(X_\eta^{\mathrm{an}}, \mathbb F_\ell) \to H^1(\hat{X}_{\eta},\mathbb F_\ell)$ induced by restriction. We thus need to see that the latter map is surjective.

By Poincar\'e duality \cite[Chapter 7]{HuberBook}, the map in question is dual to the natural map $a: H^1_c(\hat{X}_{\eta}, \mathbb F_\ell) \to H^1_c(X_\eta^{\mathrm{an}}, \mathbb F_\ell)$, so it suffices to see that $a$ is injective. Let $Y$ be the smooth projective compactification of $X_\eta$, so we have compatible open immersions $j: \hat{X}_{\eta} \to Y^{\mathrm{an}}$ and $j': X_\eta^{\mathrm{an}} \to Y^{\mathrm{an}}$. Taking cohomology on $Y^{\mathrm{an}}$ of the exact sequence $0 \to j_! \mathbb F_\ell \to j'_! \mathbb F_\ell \to (j'_! \mathbb F_\ell)/(j_! \mathbb F_\ell) \to 0$, we get an exact sequence \[0 \to H^0(Y^{\mathrm{an}},(j'_! \mathbb F_\ell)/(j_! \mathbb F_\ell)) \to H^1_c(\hat{X}_{\eta}, \mathbb F_\ell) \overset{a}{\to} H^1_c(X_\eta^{\mathrm{an}}, \mathbb F_\ell). \] However, as any connected component of $Y^{\mathrm{an}}\setminus \hat{X}_\eta$ contains a point of $Y^{\mathrm{an}}\setminus X_\eta^{\mathrm{an}}$, one has
\[
H^0(Y^{\mathrm{an}},(j'_! \mathbb F_\ell)/(j_! \mathbb F_\ell)) = 0.
\]
This gives the result.
\end{proof}

There is also a relative perverse $t$-structure on universally locally acyclic sheaves.

\begin{theorem}\label{thm:ULAmaintext} Assume that $X$ is a separated scheme of finite presentation over $S$, and consider one of the settings (B) and (C). In case (B), assume that $\Lambda$ is regular. In case (C), assume that $S$ has only finitely many irreducible components. Then there is a relative perverse $t$-structure
\[
{}^{p/S} D^{\mathrm{ULA},\leq 0}(X/S),{}^{p/S} D^{\mathrm{ULA},\geq 0}(X/S)\subset D^{\mathrm{ULA}}(X/S)
\]
such that $A\in {}^{p/S} D^{\mathrm{ULA},\leq 0}(X/S)$ (resp.~$A\in {}^{p/S} D^{\mathrm{ULA},\geq 0}(X/S)$) if and only if for all geometric points $\overline{s}\to S$, the fibre $A|_{X_{\overline{s}}}$ lies in ${}^p D^{\leq 0}(X_{\overline{s}})$ (resp.~${}^p D^{\geq 0}(X_{\overline{s}})$).
\end{theorem}

\begin{proof} In setting (B), we have to show that the truncation functors for the relative perverse $t$-structure from Theorem~\ref{thm:maintext} preserves the condition of being universally locally acyclic. As the truncation functors commute with any pullback, Corollary~\ref{cor:ULAtestrank1} reduces us to the case that $S=\Spec V$ is the spectrum of an absolutely integrally closed valuation ring of rank $1$. In that case, Theorem~\ref{thm:nearbycycles} and Lemma~\ref{lem:nearbycyclestexact} give the result.

In setting (C) with integral coefficients, one can now argue exactly as in the proof of Theorem~\ref{thm:maintext}. In setting (C) with rational coefficients, we note that by inverting $\ell$ we get the desired $t$-structure on the full subcategory
\[
D^{\mathrm{ULA}}(X/S,\mathcal O_L)[\tfrac 1\ell]\subset D^{\mathrm{ULA}}(X/S,L).
\]
Moreover, if $A\in D^{\mathrm{ULA},\leq 0}(X/S,\mathcal O_L)[\tfrac 1\ell]$ and $B\in D^{\mathrm{ULA},\geq 1}(X/S,\mathcal O_L)[\tfrac 1\ell]$, then for any scheme $S'/S$, one has $\Hom(A|_{X\times_S S'},B|_{X\times_S S'})=0$. This reduces to the case of integral coefficients, and then to torsion coefficients, arguing as in the proof of Theorem~\ref{thm:maintext} (where the pro-zeroness of some system is proved over $X$, and then follows via base change over $X\times_S S'$). This implies that for any v-cover $S'\to S$ such that $S'$ still only has finitely many irreducible components, the $t$-structure on
\[
D^{\mathrm{ULA}}(X\times_S S'/S',\mathcal O_L)[\tfrac 1\ell]\subset D^{\mathrm{ULA}}(X\times_S S'/S',L)
\]
descends to a $t$-structure on the full subcategory of $D^{\mathrm{ULA}}(X/S,L)$ of those objects whose pullback to $X\times_S S'$ admits a universally locally acyclic integral structure. Indeed, one applies the preceding observation to the $S'$-schemes $S'\times_S S'\times_S\ldots\times_S S'$ to see that the perverse truncations over $S'$ automatically descend to $S$. But by Proposition~\ref{prop:ULAintegralstructure} and \cite[Lemma 2.12]{BhattScholzeWitt}, all objects of $D^{\mathrm{ULA}}(X/S,L)$ admit such an integral structure over some finitely presented v-cover of $S$, which we can then arrange to have only finitely many irreducible components still (by replacing it by the closure of the preimage of the finitely many generic points of $S$).
\end{proof}

\begin{theorem}\label{thm:perverseULAmaintext} Fix $f:X \to S$ as usual, and consider one of the settings (B) and (C). In case (B), assume that $\Lambda$ is regular. Moreover, in all settings, assume that $S$ is irreducible, and let $\eta\in S$ be the generic point, with $j: X_\eta\subset X$ the inclusion.
\begin{enumerate}
\item[{\rm (i)}] The restriction functor
\[
j^\ast: \mathrm{Perv}^{\mathrm{ULA}}(X/S)\to \mathrm{Perv}(X_\eta)
\]
is an exact and faithful functor of abelian categories. If $\Lambda$ is noetherian, the category $\mathrm{Perv}^{\mathrm{ULA}}(X/S)$ is noetherian. If $\Lambda$ is artinian, it is also artinian.
\item[{\rm (ii)}] Assume that $S$ is geometrically unibranch. The restriction functor
\[
j^\ast: \mathrm{Perv}^{\mathrm{ULA}}(X/S)\to \mathrm{Perv}(X_\eta)
\]
is exact and fully faithful, and its image is stable under subquotients.
\end{enumerate}
\end{theorem}

\begin{remark}\label{rem:perverseULAchar0} Sometimes, one can also use any fiber in part (i). More precisely, consider setting (B), and assume that $S$ is a $\mathbb Q$-scheme that is connected (but not necessarily irreducible), and $s\in S$ is any point. Assume moreover that $X$ is proper over $S$. Let $i_s: X_s\to X$ be the inclusion of the fiber at $s$. Then
\[
i_s^\ast: \mathrm{Perv}^{\mathrm{ULA}}(X/S)\to \mathrm{Perv}(X_s)
\]
is exact and faithful.

To prove this, only faithfulness needs an argument. By noetherian approximation, we can assume that $S$ is of finite type, and then irreducible. By the theorem, it suffices to show that vanishing of the restriction to $X_s$ implies vanishing of the restriction to $X_\eta$, so we can reduce to the case that $S$ is the spectrum of an absolutely integrally closed valuation ring $V$ of rank $1$ (of equal characteristic $0$). We have to see that if $A\in \mathcal D_\cons(X_\eta,\Lambda)$ such that the nearby cycles $R\psi(A)=0$ vanish, then $A=0$. We can assume that $A$ has full support $X_\eta$, and that $X$ is normal. But then the localization of $X$ at a generic point of the special fiber is itself an absolute integrally closed valuation ring of rank $1$ (using here critically the assumption of equal characteristic $0$), and thus the stalk of $R\Psi(A)$ agrees with the generic stalk of $A$, which by assumption is nonzero.
\end{remark}

\begin{proof} In part (i), we already know that the functor is exact. We need to see that it is faithful. Once this is known, the statement that $\mathrm{Perv}^{\mathrm{ULA}}(X/S)$ is noetherian (resp.~artinian) reduces to the analogous assertion for $\mathrm{Perv}(X_\eta)$ where it is standard. Now for exact functors of abelian categories, faithfulness is equivalent to being conservative. In other words, we need to see that if $A\in \mathrm{Perv}^{\mathrm{ULA}}(X/S)$ and $j^\ast A=0$, then $A=0$. As $\eta$ specializes to any other point, we can then assume that $S$ is the spectrum of a valuation ring, and one can assume that its fraction field is algebraically closed. Then the result follows from Theorem~\ref{thm:nearbycycles}.

In part (ii), we already know that the functor is exact and faithful. Consider first setting (B). This can be embedded into setting (A), and we first claim that for any $A\in \mathrm{Perv}^{\mathrm{ULA}}(X/S)$, the map
\[
A\to {}^{p/S}\tau^{\leq 0} Rj_\ast j^\ast A
\]
is an isomorphism. In fact, being universally locally acyclic implies that
\[
Rj_\ast j^\ast A\cong A\dotimes_{\Lambda} f^\ast(Rk_\ast \Lambda)
\]
(as in the proof of Proposition~\ref{prop:basicpropertiesULA}) where $k: \eta\subset \Spec S$ is the inclusion. Now it follows from the cone of $M\to Rk_\ast M$ being in degrees $\geq 1$ for any $\ell$-power torsion $\Lambda$-module $M$, which is a simple consequence of being geometrically unibranch. The map $A\to {}^{p/S}\tau^{\leq 0} Rj_\ast j^\ast A$ being an isomorphism implies that $j^\ast: \mathrm{Perv}^{\mathrm{ULA}}(X/S)\to \mathrm{Perv}(X_\eta)$ is fully faithful.

In setting (B), it remains to see that the image is stable under passage to subquotients. It is enough to handle subobjects, so take $A\in \mathrm{Perv}^{\mathrm{ULA}}(X/S)$ and let $B_0\subset j^\ast A\in \mathrm{Perv}(X_\eta)$ be a subobject. First, we show that if $S'\to S$ is a projective birational map such that $B_0$ admits an extension to $B'\in \mathrm{Perv}^{\mathrm{ULA}}(X_{S'}/S')$, then $B_0$ even extends to $B\in \mathrm{Perv}^{\mathrm{ULA}}(X/S)$. By v-descent, it suffices to see that the two pullbacks of $B'\subset A|_{X_{S'}}$ to $X_{S'\times_S S'}$ agree (as sub-perverse sheaves of $A|_{X_{S'\times_S S'}}$). They clearly agree when restricted to the diagonal $S'\subset S'\times_S S'$. But each geometric fibre $S'_{\overline{s}}$ of $S'\to S$, over a geometric point $\overline{s}\to S$, is a connected projective variety (as $S$ is geometrically unibranch), and thus by Lemma~\ref{lem:familiessubperversesheaves} the restriction of $B'$ to $X_{S'_{\overline{s}}}$ must be a constant sub-perverse sheaf of $A|_{X_{\overline{s}}}$ base-changed to $S'_{\overline{s}}$. This gives the desired claim.

For any such $S'\to S$, we can look at the maximal open subscheme $U'\subset S'$ to which $B_0$ extends as a universally locally acyclic perverse sheaf. (Here, as an exception, $U'$ may not be quasicompact.) Assume that $U'\neq S'$ for all such $S'\to S$. Then we can find a compatible family of points in $S'\setminus U'$ over all $S'\to S$, giving in the inverse limit a valuation ring $\Spec V\to S$ with $\Spec K=\eta\subset S$, where $K$ is the fraction field of $V$, and by Proposition~\ref{prop:ULAfinitaryarc} the non-existence of an extension of $B_0$ to a universally locally acyclic (necessarily perverse) sheaf to $S'\setminus U'$ implies that there is no such extension to $\Spec V$ either. In other words, we can assume $S=\Spec V$ is the spectrum of a valuation ring. We can now similarly pass up the tower of finite covers of $V$ (noting that taking generically \'etale extensions with Galois group $G$, any extension will automatically be $G$-equivariant and hence descend; while inseparable extensions do not matter). Thus, we can assume that the fraction field of $V$ is algebraically closed. But now Theorem~\ref{thm:nearbycycles} shows that $B_0$ must extend (and necessarily to a sub-relatively perverse sheaf, by Lemma~\ref{lem:nearbycyclestexact}).

It remains to prove (ii) in setting (C). With integral coefficients, this reduces easily to setting (B). To deduce it with rational coefficients, it suffices to show that any $A\in \mathrm{Perv}^{\mathrm{ULA}}(X/S,L)$ admits an $\ell$-torsion free integral structure $A_0\in \mathrm{Perv}^{\mathrm{ULA}}(X/S,\mathcal O_L)$. In fact, such integral structures are equivalent to $\ell$-torsion free integral structures of $A_\eta$ (which, over a field, are automatically universally locally acyclic). It follows from the case of integral coefficients that such an integral structure $A_0$ of $A$ is determined by the integral structure of $A_\eta$ (i.e., the forgetful functor is fully faithful); to see that it is essentially surjective, we can argue as in the previous two paragraphs, using the second part of Lemma~\ref{lem:familiessubperversesheaves}.
\end{proof}

We used the following lemma.

\begin{lemma}\label{lem:familiessubperversesheaves} Let $k$ be an algebraically closed field, let $X/k$ be a separated scheme of finite type, let $\Lambda$ be a regular $\mathbb Z_\ell$-algebra and let $A\in \mathrm{Perv}(X,\Lambda)$ in setting (B). The functor taking a $k$-scheme $S$ to the set of universally locally acyclic sub-relative perverse sheaves $B\subset A|_{X_S}$ in $\mathrm{Perv}^{\mathrm{ULA}}(X_S/S)$ is representable by a $k$-scheme that is a disjoint union of copies of $\Spec k$.

Similarly, if $A\in \mathrm{Perv}(X,L)$ in setting (C), then the functor taking any $k$-scheme $S$ to the set of universally locally acyclic $A_0\in \mathcal D_\cons(X_S,\mathcal O_L)$ with $A_0[\tfrac 1\ell]\cong A|_{X_S}$ and such that $A_0/^{\mathbb L}\ell\in \mathcal D_\cons(X_S,\mathcal O_L/\ell)$ is relatively perverse, is representable by a $k$-scheme that is a disjoint union of copies of $\Spec k$.
\end{lemma}

\begin{proof} In both cases, we need to see that this functor is the constant sheaf on its value on $S=\Spec k$. By adjunction, there is a map, and both functors are finitary arc-sheaves. It is thus sufficient to show that it induces an isomorphism on $S=\Spec V$-valued points where $V$ is an absolutely integrally closed valuation ring over $k$. By Theorem~\ref{thm:nearbycycles} and Lemma~\ref{lem:nearbycyclestexact}, one can reduce to the generic fibre $K$ of $V$. Now it is a simple consequence of general properties of invariance under change of algebraically closed base field. Indeed, in the first setting one can filter $A$ by intermediate extensions of local systems on (smooth) strata to reduce to the case of local systems on smooth $X$. In that case $B$ is also necessarily a local system, and the result follows from $\pi_1(X_K)\to \pi_1(X)$ being surjective. A similar argument works in the second setting.
\end{proof}

Finally, we note that the results also give the following result.

\begin{proposition}\label{prop:ULAintegralstructureunibranch} Assume that $S$ is geometrically unibranch and irreducible. Let $f: X\to S$ be a separated scheme of finite presentation and $A\in D_\cons^{\mathrm{ULA}}(X,L)$ in setting (C) with rational coefficients. Then there is some $A_0\in D_\cons^{\mathrm{ULA}}(X,\mathcal O_L)$ with $A\cong A_0[\tfrac 1\ell]$. If $A$ is relatively perverse, one can find such an $A_0$ that is also relatively perverse and $\ell$-torsion free (as a relatively perverse sheaf).
\end{proposition}

\begin{proof} Passing to a filtration of $A$, we can assume that $A$ is relatively perverse. In that case, there is an $A_0$ that is relatively perverse and $\ell$-torsion free, as was proved at the end of the proof of Theorem~\ref{thm:perverseULAmaintext}.
\end{proof}

\bibliographystyle{amsalpha} 
\providecommand{\bysame}{\leavevmode\hbox to3em{\hrulefill}\thinspace}
\providecommand{\MR}{\relax\ifhmode\unskip\space\fi MR }
\providecommand{\MRhref}[2]{%
  \href{http://www.ams.org/mathscinet-getitem?mr=#1}{#2}
}
\providecommand{\href}[2]{#2}

\end{document}